\documentclass[11pt]{article}
\author{Richard Olu Awonusika\footnote{Corresponding author}  \hspace{0.8cm} Olasupo John Felemu	\\
	\texttt{richard.awonusika@aaua.edu.ng \hspace{0.8cm} olasupo.felemu@aaua.edu.ng
	}\\
	Olawale Olaonipekun Ajijola\hspace{0.8cm} Yoyinade Joose Aborisade	\\
	\texttt{olawale.ajijola@aaua.edu.ng} \hspace{0.8cm} \texttt{yoyinade.aborisade@aaua.edu.ng}\\
	Department of Mathematical Sciences\\ Adekunle Ajasin University, P.M.B. 001, Akungba-Akoko, Ondo State, Nigeria\\
}
\title{\textbf{Numerical Solution of Pantograph Delay Integrodifferential Equation of Volterra Type: Collocation Method Based on Shifted Jacobi Polynomials}}
\date{}

\usepackage{mathrsfs}
\usepackage{lscape}
\usepackage{array}
\usepackage{amsthm}
\usepackage{amsmath}
\usepackage{amssymb}
\usepackage[margin=1.25 in]{geometry}
\newtheorem{theorem}{Theorem}[section]
\newtheorem{example}{Example}[section]
\newtheorem{lemma}{Lemma}[section]
\theoremstyle{definition}
\theoremstyle{example}

\theoremstyle{remark}

\numberwithin{equation}{section}
\numberwithin{theorem}{section}
\usepackage{hyperref}

\usepackage{enumerate}
\usepackage{enumerate}
\usepackage{color}
\usepackage[square, numbers, comma, sort&compress]{natbib}
\usepackage{float}
\newmuskip\pFqmuskip

\newcommand*\pFq[6][8]{%
	\begingroup 
	\pFqmuskip=#1mu\relax
	\mathchardef\normalcomma=\mathcode`,
	\mathcode`\,=\string"8000
	\begingroup\lccode`\~=`\,
	\lowercase{\endgroup\let~}\pFqcomma
	{}_{#2}F_{#3}{\left[\genfrac..{0pt}{}{#4}{#5};#6\right]}%
	\endgroup
}
\newcommand{\pFqcomma}{{\normalcomma}\mskip\pFqmuskip}

\begin{document}
	\maketitle
	%

\begin{abstract}
Pantograph arises in electric trains, material modelling, and the modelling of quantum dot lasers. Pantograph integrodifferential equations are integrodifferential equations involving proportional delays; and they appear in fields such as  electrodynamics, epidemiology, control theory, astrophysics, economics,  and engineering. This research presents an efficient collocation method based on shifted Jacobi polynomials for obtaining numerical solutions of a class of first order pantograph delay integrodifferential equation of Volterra type with an initial condition. The proposed method expresses the solution of the governing equation as a shifted Jacobi polynomial series with expansion coefficients which are to be determined.   Collocating at the roots of the shifted Jacobi polynomials, the underlying problem is reduced to a system of
algebraic equations in the unknown expansion coefficients of the shifted Jacobi polynomial series solution. Newton's method is subsequently used to solve the resulting system of algebraic equations and numerical values of the expansion coefficients are obtained. The obtained coefficients are substituted into the assumed series solution to obtain the required numerical solutions. The applicability, reliability, efficiency, and accuracy of the shifted Jacobi collocation method are demonstrated through illustrative examples. Results obtained using the proposed method are compared with exact solutions and other published results. Comparisons of errors which are presented in tables reveal that our method approximates the solution better than the methods under comparison. 
\end{abstract}

	\section{Introduction}
The pantograph differential equation in the form of a special delay first order ordinary differential equation was first investigated by Ockendon and Tayler \cite{Ock}:
\begin{equation}\label{eqP}
\phi'(x)=a\phi(x)+b\phi(qx),
 \end{equation}
where $a,b$ are real constants and $0<q< 1$. The term $\phi(qx)$ depicts that the delay is not constant, but proportional to the current $x$, with  $0<q< 1$. This form of pantograph equation arises in areas that include epidemiology, electrodynamics,  engineering, and economics \cite{Adel}, \cite{Vanani2}. Pantograph was initially introduced in scaling and drawing, and were further discussed in the modelling of lasers, especially quantum dot lasers, material modelling, and electric trains (\cite{Ahmad2021}, \cite{Aourir2025}, \cite{Behera2022}). Unlike standard ordinary differential equations, the simplest linear pantograph equations cannot be solved in terms of elementary functions. For several analytical and numerical methods of solving \eqref{eqP}, see  \cite{Bhalekar},  \cite{Iserles}, \cite{Kato}, \cite{Liu},  \cite{Pat}, \cite{Yang}.   Lane-Emden type equations involving proportional delays were considered in   \cite{Adel}, \cite{Alsaw}, \cite{AwoMogbodelay}, \cite{Azin2021}, \cite{Elbrahimi2019}, \cite{14Izadi2021}, \cite{Nuw},   \cite{Rabiei},  \cite{Rah}, \cite{Shaikh2019},  \cite{Vichitkunakorn}, \cite{Yuttanan}.

Integral equations arise in the modeling of several real-life phenomena in mathematical biology, chemistry, electricity, elasticity, magnetism, mathematical economics, epidemiology,  population growth theory, fluid mechanics,  mathematical physics, engineering, potential theory, acoustics, 
renewal theory,  and geophysics (\cite{Bruner2017}, \cite{Hethcote1980},  \cite{Miller}, \cite{Pettitt1982}, \cite{Ting1981}).  Integral equations are basically classified into two -  \textit{Fredholm} integral equations (\cite[Ch. 1]{Zemyan2012}) and \textit{Volterra} integral equations (\cite[Ch. 4]{Zemyan2012}). An integral equation having constant limits of integration is called a Fredholm integral equation, while the one with  variable limits of integration is known as a Volterra integral equation. Integral equations involving the derivative(s) of the unknown variable  are known as integrodifferential equations (\cite[Ch. 5]{Zemyan2012}).
 One of the main areas of integral equations which have received the most attention from researchers in recent years is  pantograph delay integrodifferential equation. Pantograph delay integrodifferential equations are a class of functional integral equations featuring proportional delays, while the (Volterra or Fredholm) integral terms model systems with memory, feedback, and scaling, such as biological evolution, population dynamics, and control theory  \cite{3Dehghan}, \cite{2Ezz-Eldien},  \cite{4Fox}, \cite{1Kuang}. 
 
 Pantograph integrodifferential equations  are generally considered as integrodifferential equations involving proportional delays. Pantograph delay integrodifferential equations of Volterra type arise in fields such as astrophysics, control theory, biology,  electrodynamics, and economics (\cite{Buhmann}). These features
 ensure the relevance of the pantograph integrodifferential equation in modeling the dynamics of populations according to the age distribution and processes of biological evolution, where the current behaviour of the system depends on its historical behaviours at varying scales. Furthermore, pantograph integrodifferential equations can help optimise and design systems that involve scaling and delayed feedback such as communication and robots systems \cite{7Ajello},  \cite{9Alsuyuti}, \cite{10Bocharov}, \cite{8Elkot}, \cite{6Zaky}. The presence of a pantograph term in the integrodifferential equation makes the process of obtaining its exact solution difficult, thus seeking efficient and reliable methods for obtaining their numerical solutions becomes the main goal of researchers in the field.  
 
 Several analytical and numerical methods, such as tau method, Taylor polynomial method, Laguerre matrix method, spectral tau method, the Hermite polynomial method, Adomian decomposition method, operational tau method, and Haar wavelet method have been used to obtain solutions of pantograph integrodifferential equations (\cite{29Sezer2008}, \cite{27Shakeri2010},  \cite{26Yuzbası2015}). Recently, in \cite{Behera2022}, Behera and Ray applied  Euler wavelets to obtain numerical solutions of pantograph Volterra delay integrodifferential equations.  Jacobi tau spectral approach was used in \cite{Ezz-Eldien2025} to solve multi-pantograph integrodifferential equations. In \cite{Ahmad2021}, Haar Wavelets method was used to solve pantograph delay integrodifferential equations.  In \cite{11Zhao}, the Sinc collocation approach
 for solving pantograph delay integrodifferential equations was employed. The Galerkin spectral technique
was applied in \cite{12Alsuyuti} and \cite{13Zaky} to solve a class of partial differential and integral
equations with arbitrary orders. The authors in \cite{14Zhao} used Bernstein tau spectral  and Lagrange interpolation
collocation methods for the solution of  pantograph delay integrodifferential
equations. 

In \cite{15Zaky}, the authors implemented the multivariate Jacobi approximation collocation technique for multi-dimensional weakly singular nonlinear integral equations with nonsmooth
solutions.  Ji et al. \cite{21Ji} derived derivative and
pantograph operational matrices for Chebyshev spectral collocation
method  to solve pantograph integrodifferential equations. Using the operational technique
based on Legendre polynomials, Jafari et al. \cite{22Jafari} solved pantograph nonlinear
integrodifferential equations. Wang et al. \cite{23Wang} employed an operational procedure based on Jacobi polynomials  for a high-order
nonlinear two-dimensional integrodifferential equation. Darania and
Sotoudehmaram \cite{24Darania} applied the multistep collocation technique for numerical solutions of nonlinear
delay integral equations. For the bilinear neural network approach and neural network-based symbolic techniques for various types of partial integral
and pantograph-type delay partial differential equations, see \cite{19Alkhezi}, \cite{25Elkot}, \cite{16Liu},    \cite{26Tedjani},  \cite{20Xie}, \cite{27Zaky}, \cite{17Zhang}, \cite{18Zhong}.

In this research, we use shifted Jacobi collocation method   to obtain numerical solutions of the pantograph delay first order integrodifferential equation of Volterra type	
\begin{align}
	\phi'(x)=&f(x)+ a_{1}(x)\phi(x)+a_{2}(x)\phi( \mu x)+\int_0^x K_{1}(x,s)u(s)\, ds+\int_0^{\mu x} K_{2}(x,s)u(s) \, ds,\label{eq main 1}\\
	\phi(0)=&\gamma.\label{eq main 1IC}
\end{align}
 Here $\gamma$ is a constant; $x\in I:=[0,X]$; $\mu\in(0,1)$; the functions $K_{1}(x,s)$ defined on $\mathcal{R}:=\left\lbrace (x,s)\in I\times I \right\rbrace $ and $K_{2}(x,s)$  defined on $\widetilde{\mathcal{R}}:=\left\lbrace (x,s)\in I\times [0,\mu x] \right\rbrace $ are given; $a_{j}$ $(j=1,2)$ are given bounded functions; and $f:[0,\infty)\rightarrow \mathbb{R}$. In the proposed method, one assumes that the solution of \eqref{eq main 1}-\eqref{eq main 1IC} can be expressed in the form of a shifted Jacobi polynomial series. Upon substituting the assumed series solution into  \eqref{eq main 1}-\eqref{eq main 1IC}, one is required to apply a differentiation formula for shifted Jacobi polynomials.  Collocating at the zeros of the shifted Jacobi polynomials, we get a set of algebraic equations, which are subsequently solved for the expansion coefficients of the shifted Jacobi polynomial series via Newton's iteration technique. The obtained values of these coefficients are in turn substituted into the assumed polynomial series solution to obtain the required numerical solutions.  The convergence rate of shifted Jacobi polynomial series for integrodifferential equations of Volterra type is discussed. Examples of the governing equation are presented to illustrate the proposed method's reliability, effectiveness, and accuracy. The obtained numerical solutions are compared with the exact solutions and other published results.  Comparisons of results are demonstrated in tables. Our proposed collocation scheme yields more accurate results compared to other existing methods, which is a clear indication that the proposed method is reliable, efficient, and accurate for solving pantograph delay integrodifferential equations of Volterra type. 

%

\section{Preliminaries}\label{App Jac}

\subsection{Jacobi Polynomials}\label{Sec ClasJac}	
The classical $\ell$th-degree Jacobi polynomials $J_{\ell}^{(\alpha,\beta)}=J_{\ell}^{(\alpha,\beta)}(t)$ ($\ell=0,1,2,\dots$; $\alpha,\beta> -1;t\in[-1,1]$) are solutions of the Jacobi  differential equation ($y:=J_{\ell}^{(\alpha, \beta)}(t)$)
\begin{equation}\label{eq jap2}
	\left( 1-t^{2}\right) \frac{d^2y}{dt^2} - (\alpha-\beta+(\alpha+\beta+2)t) \frac{dy}{dt} + \ell(\ell+\alpha+\beta+1) y =0, 
\end{equation}
and they admit the generating function  
\begin{equation}\label{eq p1}
	2^{\alpha+\beta}\Lambda^{-1}\left(1-z+\Lambda\right)^{-\alpha}\left(1+z+\Lambda\right)^{-\beta}=\sum\limits_{k=0}^{\infty}J_{k}^{(\alpha,\beta)}(t)z^{k}, \,\, \Lambda=\sqrt{1-2t z+z^{2}},\,|z|<1.
\end{equation}
Jacobi polynomials $J_{\ell}^{(\alpha,\beta)}(t)$ have the finite  series formulation
\begin{align}\label{eq po ser}
	J_{\ell}^{(\alpha,\beta)}(t)=&\frac{\Gamma(\ell+\alpha+1)}{\Gamma(\ell+\alpha+\beta+1)}\sum_{k=0}^{\ell}{\ell\choose k}\frac{\Gamma(\ell+\alpha+\beta+k+1)}{2^{k}\Gamma(\alpha+k+1)\ell!}(t-1)^{k}
	=\sum_{k=0}^{\ell}F_{\ell,k}^{\alpha,\beta}\left( \frac{1+t}{2}\right) ^{k},
\end{align}
where
\begin{equation}\label{eq ecjac}
	F_{\ell,k}^{\alpha,\beta}:=\frac{(-1)^{\ell-k}\Gamma(\ell+\beta+1)\Gamma(\ell+\alpha+\beta+k+1)}{\Gamma(\ell+\alpha+\beta+1)\Gamma(\beta+k+1)(\ell-k)!k!};
\end{equation}
and satisfying  the differentiation  formula
\begin{equation}\label{eqc6}
	\frac{d^{p}}{dt^{p}}J_{\ell}^{(\alpha,\beta)}(t)=\frac{1}{2^{p}}\frac{\Gamma(\ell+p+\alpha+\beta+1)}{\Gamma(\ell+\alpha+\beta+1)} J_{\ell-p}^{(\alpha+p,\beta+p)}(t) \qquad (p=1,2,\dots).
\end{equation}
The following reflection symmetry and pointwise identities hold.
\begin{equation}\label{jap8}
J_{\ell}^{(\alpha,\beta)}(-t)=(-1)^{\ell}J_{\ell}^{(\beta,\alpha)}(t), \quad J_{\ell}^{(\alpha,\beta)}(1)=\frac{\Gamma(\ell+\alpha+1)}{\ell!\Gamma(\alpha+1)}.
\end{equation}
As orthogonal polynomials, they admit the integral formula 
\begin{align}
	&\int_{-1}^{1}J_{\ell}^{(\alpha,\beta)}(t)J_{m}^{(\alpha,\beta)}(t)\left( 1-t\right)^{\alpha}\left( 1+t\right)^{\beta}\,dt
	=\kappa^{\alpha,\beta}_{\ell}\delta_{\ell m} \qquad (\ell,m=0,1,2,\dots),
\end{align}
where  $\kappa^{\alpha,\beta}_{\ell}$ are given by 
\begin{equation}\label{eq eta}
	\kappa^{\alpha,\beta}_{\ell}= \frac{2^{\alpha+\beta+1}}{2\ell+\alpha+\beta+1}\frac{\Gamma(\ell+\alpha+1)\Gamma(\ell+\beta+1)}{\Gamma(\ell+1)\Gamma(\ell+\alpha+\beta+1)}.
\end{equation}
Legendre polynomials, Chebyshev polynomials of the first, second, third, and fourth kinds are special cases of Jacobi polynomials with the following relations.
\begin{align}\label{jap4}
&P_{\ell}(t):=J_{\ell}^{(0,0)}(t) \qquad \mbox{(Legendre polynomials)}\\
&T_{\ell}(t): = \frac{\ell!\Gamma\left( \frac{1}{2}\right) }{\Gamma\left( \ell+\frac{1}{2}\right) }J_{\ell}^{\left( -\frac{1}{2},-\frac{1}{2}\right) }(t)\qquad \mbox{(Chebyshev polynomials of the first kind)}\\
&U_{\ell}(t):=\frac{\ell!\Gamma\left( \frac{3}{2}\right) }{\Gamma\left( \ell+\frac{3}{2}\right) }J_{\ell}^{\left( \frac{1}{2},\frac{1}{2}\right) }(t) \qquad \mbox{(Chebyshev polynomials of the second kind)}\\
&V_{\ell}(t):= \frac{\ell!\Gamma\left( \frac{1}{2}\right) }{\Gamma\left( \ell+\frac{1}{2}\right) }J_{\ell}^{\left( -\frac{1}{2},\frac{1}{2}\right) }(t)\qquad \mbox{(Chebyshev polynomials of the third kind)}\\
&W_{\ell}(t):= \frac{\ell!\Gamma\left( \frac{3}{2}\right) }{\Gamma\left( \ell+\frac{3}{2}\right) }J_{\ell}^{\left( \frac{1}{2},-\frac{1}{2}\right) }(t)\qquad \mbox{(Chebyshev polynomials of the fourth kind)}.
\end{align}

\subsection{Shifted Jacobi Polynomials}\label{Sec ShifJac}	
The $\ell$th-degree shifted Jacobi polynomials
$\mathsf{J}_{\ell}^{(\alpha,\beta)}(t)$ $(t\in[0,1])$ are defined by the finite  series representation
\begin{align}\label{eq po ser1}
	\mathsf{J}_{\ell}^{(\alpha,\beta)}(t):=
	J_{\ell}^{(\alpha,\beta)}\left( 2t-1\right) 
	=&\sum_{k=0}^{\ell}F_{\ell,k}^{\alpha,\beta}t^{k},
\end{align}
where
$F_{\ell,k}^{\alpha,\beta}$ is as given in \eqref{eq ecjac}. In particular, we have the following special values.
\begin{align}\label{eq sifpo}
	\mathsf{J}_{\ell}^{(\alpha,\beta)}(0)=(-1)^{\ell}\frac{\Gamma(\ell+\beta+1)}{\Gamma(\beta+1)\ell!}, \qquad \mathsf{J}_{\ell}^{(\alpha,\beta)}(1)=\frac{\Gamma(\ell+\alpha+1)}{\Gamma(\alpha+1)\ell!}.                                                                                    
\end{align}
Furthermore, the following differentiation formulae hold $(r=0,1,2,\dots)$.
\begin{align}
	\frac{d^{r}}{dt^{r}}	\mathsf{J}_{\ell}^{(\alpha,\beta)}(0)=&\frac{(-1)^{\ell-r}\Gamma(\ell+\beta+1)\Gamma(\ell+r+\alpha+\beta+1)}{\Gamma(r+\beta+1)\Gamma(\ell+\alpha+\beta+1)\Gamma(\ell-r+1)}\label{eq diff shifted1}\\
	\frac{d^{r}}{dt^{r}}	\mathsf{J}_{\ell}^{(\alpha,\beta)}(1)=&\frac{\Gamma(\ell+\alpha+1)\Gamma(\ell+r+\alpha+\beta+1)}{\Gamma(r+\alpha+1)\Gamma(\ell+\alpha+\beta+1)\Gamma(\ell-r+1)}\label{eq diff shifted-1}\\
	\frac{d^{r}}{dt^{r}}	\mathsf{J}_{\ell}^{(\alpha,\beta)}(t)=&\frac{\Gamma(\ell+r+\alpha+\beta+1)}{\Gamma(\ell+\alpha+\beta+1)}	\mathsf{J}_{\ell-r}^{(\alpha+r,\beta+r)}(t).
	\label{eq diff shifted}
\end{align}

Consider the weight function $w^{\alpha,\beta}=w^{(\alpha,\beta)}(t)=(1-t)^{\alpha}t^{\beta}$. We define the inner product and norm in the weighted space 
\begin{equation}
	L^{2}_{w^{\alpha,\beta}}([0,1])=\left\lbrace \phi: \phi \mbox{ is measurable and } \left\| \phi\right\|_{w^{\alpha,\beta}}<\infty\right\rbrace,
\end{equation}
 respectively, by
\begin{align}
	(\phi,\varphi)_{w^{\alpha,\beta}}=\int_{0}^{1}\phi(t)\varphi(t)w^{(\alpha,\beta)}(t)\,dt, \qquad \left\| \phi\right\|_{w^{\alpha,\beta}}=(\phi,\phi)^{1/2}_{w^{\alpha,\beta}}. 
\end{align}
A complete $L^{2}_{w^{\alpha,\beta}}[0,1]$-orthogonal system consists of a set of shifted Jacobi polynomials, namely,
\begin{align}
	&\int_{0}^{1}\mathsf{J}_{\ell}^{(\alpha,\beta)}(t)\mathsf{J}_{m}^{(\alpha,\beta)}(t)w^{(\alpha,\beta)}(t)\,dt
	=\left( \frac{1}{2}\right)^{\alpha+\beta+1} \kappa^{\alpha,\beta}_{\ell}\delta_{\ell m}	
	\qquad (\ell,m=0,1,2,\dots),
\end{align}
where the scalars $\kappa^{\alpha,\beta}_{\ell}$ are as given in \eqref{eq eta}.

Let the roots of the Jacobi polynomials $J_{N+1}^{(\alpha,\beta)}(t)$  in the interval $[-1,1]$ be denoted by $t^{N}_{k}(\alpha,\beta)$ $(0\leq k\leq N)$. The roots $\mathsf{t}^{N}_{k}(\alpha,\beta)$ $(0\leq k\leq N)$ of the shifted Jacobi polynomials $\mathsf{J}_{N+1}^{(\alpha,\beta)}(t)$ in the interval $[0,1]$ are given by
\begin{align}
	\mathsf{t}^{N}_{k}(\alpha,\beta)=\frac{1}{2}\left( t^{N}_{k}(\alpha,\beta)+1\right).
\end{align}

\section{Shifted Jacobi Collocation Method of Solution}\label{Sec Collo Ago}	
This section  presents shifted Jacobi collocation algorithm for obtaining numerical solutions of the pantograph delay Volterra integrodifferential equation \eqref{eq main 1} with the initial condition  \eqref{eq main 1IC}. The method assumes that the function $\phi(x)$ admits a shifted Jacobi polynomial series formulation, with expansion coefficients to be determined. To obtain these   expansion coefficients, we collocate at zeros of the shifted Jacobi polynomials satisfying the required initial conditions, to get a set of algebraic equations which  is solved for the unknown expansion coefficients using Newton's method. 


To this end, the shifted Jacobi collocation method assumes that  the solution $\phi(x)$ of the governing problem \eqref{eq main 1}
can be expressed as shifted Jacobi polynomial series
\begin{align}\label{eqsoluv}
	\begin{split}
\phi(x)\approxeq	\phi_{N}(x)=&\sum_{k=0}^{N}c_{k}\mathsf{J}^{(\alpha,\beta)}_{k}(x),
	\end{split}
\end{align}
satisfying the initial condition
\begin{align}\label{eqsolIC}
	\begin{split}
		\phi_{N}(0)=&\sum_{k=0}^{N}c_{k}(-1)^{k}\frac{\Gamma(k+\beta+1)}{\Gamma(\beta+1)k!}=\gamma.
	\end{split}
\end{align}

We now give a detailed step-by-step procedure in the shifted Jacobi collocation  method of solving the proposed integrodifferential problem \eqref{eq main 1}-\eqref{eq main 1IC} as follows: 

\begin{enumerate}[Step I]
	
	\item Given the pantograph delay integrodifferential  problem \eqref{eq main 1}-\eqref{eq main 1IC}.
	
	\item\label{step cm} Assume that the solution $\phi(x)$   $(0<x\leq 1)$ admits the shifted Jacobi polynomial series with the  expansion coefficients $c_{k}\in \mathbb{R}$ as given in \eqref{eqsoluv}. 
	
	\item Substitute the series formulation \eqref{eqsoluv} into the problem \eqref{eq main 1}-\eqref{eq main 1IC}.

	\item Let $\mathsf{x}^{N}_{q}(\alpha,\beta)$ $(0\leq q\leq N)$ denote the zeros of the  shifted Jacobi polynomials $\mathsf{J}^{(\alpha,\beta)}_{N+1}(x)$.
	
	\item\label{Step colsc} Collocate at the zeros $\mathsf{x}^{N-1}_{q}(\alpha,\beta)$ of the shifted Jacobi polynomials $\mathsf{J}^{(\alpha,\beta)}_{N}\left( x\right)$; thus the integral problem  \eqref{eq main 1}-\eqref{eq main 1IC} is satisfied exactly at these collocation points. The resulting set of equations forms a collocation scheme.
	\item\label{step colls}  The collocation scheme obtained in Step \ref{Step colsc}  gives a set of $N+1$  algebraic equations in the expansion coefficients $c_{k}$ $(0\leq k\leq N)$. 
	\item \label{step coe ck} Solve the algebraic equations in Step \ref{step colls} for the coefficients  $c_{k}$ $(0\leq k\leq N)$ using Newton's method (via Wolfram Mathematica 12.0).	
	\item Substitute the expansion coefficients  $c_{k}$ $(0\leq k\leq N)$ obtained in Step  \ref{step coe ck} into the assumed shifted Jacobi polynomial series solution \eqref{eqsoluv} to obtain the required approximate solution ($N$th-degree polynomial).	
	\item Use the approximating polynomial to perform the necessary numerical simulations.  
\end{enumerate}

\section{Convergence and Error Analysis}
In this section, we discuss the convergence and error estimates of the proposed shifted Jacobi collocation method for solving the governing pantograph delay integrodifferential equation of Volterra type \eqref{eq main 1}-\eqref{eq main 1IC}. Notable authors have presented the convergence and error analysis of shifted Jacobi collocation method for Volterra integral equations of the second kind, Volterra-Fredholm integral equations of the second kind, and Volterra integrodifferential equations of the second kind. Precisely, in \cite{Amin2023}, Amin et al. presented extensively the $L^{2}_{w^{\alpha,\beta}}[0,1]$-norm convergence and error analysis of shifted Jacobi collocation method for Volterra-Fredholm integral equations of the second kind. Guo et al. \cite{Guo2014} gave robust $L^{2}_{w^{\alpha,\beta}}[0,1]$-norm and $L^{\infty}[0,1]$-norm convergence analyses of Jacobi collocation method for Volterra integral equations of the second kind with a smooth kernel. Ma and Huang \cite{Ma2014} accounted for the convergence rate of shifted Jacobi polynomial series for linear Volterra integrodifferential equations of the second kind. This section accounts for the convergence and error estimates of the  pantograph delay Volterra integrodifferential equation of the second kind \eqref{eq main 1}-\eqref{eq main 1IC}. 

%

\subsection{Preliminaries}

Let $\mathcal{J}_{N}$ denote the space of all  polynomials of degree not exceeding $N$. Define the orthogonal projection operator $\pi_{N,w^{(\alpha,\beta)}}:L^{2}_{w^{(\alpha,\beta)}}([0,1])\rightarrow\mathcal{J}_{N}$ by (\cite{Amin2023}, \cite{Ma2014})
\begin{equation}
	\left(\phi-\pi_{N,w^{(\alpha,\beta)}}\phi, \varphi \right) =0, \qquad \forall \phi\in L^{2}_{w^{(\alpha,\beta)}}([0,1]), \varphi\in\mathcal{J}_{N}.
\end{equation}

For $m=0,1,2,3,\cdots$, define
\begin{equation}
H^{m}_{w^{(\alpha,\beta)}}([0,1])=\left\lbrace u: \,\partial_{x}^{j}u\in L^{2}_{w^{(\alpha,\beta)}}([0,1]), \quad 0\leq j\leq m \right\rbrace, 
\end{equation}
with the semi norm and norm defined, respectively, by
\begin{equation}
u_{m,w^{(\alpha,\beta)}}=\left\| \partial_{x}^{m}u\right\|, \qquad \left\| u\right\|_{H^{m}_{w^{(\alpha,\beta)}}([0,1])} =\left( \sum_{j=0}^{m}\left\| \partial_{x}^{j}u\right\|_{w^{(\alpha,\beta)}}^{2}\right)^{\frac{1}{2}}.   
\end{equation}

Now assume that $\phi\in H^{m}_{w^{(\alpha,\beta)}}([0,1])$, and let $I^{\alpha,\beta}_{N}\phi\in\pi_{N,w^{(\alpha,\beta)}}$ denote the interpolation of $\phi$ at the Jacobi-Gauss nodes. We now give the following preliminary results.

\begin{lemma}[\cite{Amin2023}, \cite{Ma2014}]\label{le 1}
For any function $\phi\in H^{m}_{w^{(\alpha,\beta)}}([0,1])$, $I^{\alpha,\beta}_{N}\phi\in\pi_{N,w^{(\alpha,\beta)}}$, we have
\begin{equation}
\left\| \phi-I_{N}^{\alpha,\beta}\phi \right\|_{L^{2}_{w^{(\alpha,\beta)}}([0,1])} \leq CN^{-m}\left\| \phi\right\|_{H^{m}_{w^{(\alpha,\beta)}}([0,1])} \qquad (\alpha,\beta>-1),
\end{equation}	
\begin{equation}
	\left\| \phi-I_{N}^{\alpha,\beta}\phi \right\|_{L^{\infty}([0,1])} \leq CN^{\frac{1}{2}-m}\left\| \phi\right\|_{H^{m}_{w^{(\alpha,\beta)}}([0,1])} \qquad (\alpha,\beta>-1),
\end{equation}	
where $C$ is a positive constant that is independent of $N$. 
\end{lemma}
%
%
%
%
%
%

The following result is well-known.
\begin{lemma}[Cauchy--Schwartz Inequality]\label{le CSIneq}
	Let $f,g\in L^2(0,1)$. Then
	\[
	\left|
	\int_0^1 f(x)g(x)\,dx
	\right|
	\le
	\left(
	\int_0^1 |f(x)|^2dx
	\right)^{1/2}
	\left(
	\int_0^1 |g(x)|^2dx
	\right)^{1/2}.
	\]
	
\end{lemma}
%
%
%
%
%
%
%

%
%

%
%
%
%
%
%
%
%

%
%
%

\subsection{Convergence and Error Estimates}

Consider the pantograph delay Volterra integrodifferential equation
\begin{align}
	\phi'(x)=&f(x)+ a_{1}(x)\phi(x)+a_{2}(x)\phi( \mu x)+\int_0^x K_{1}(x,s)\phi(s)\, ds+\int_0^{\mu x} K_{2}(x,s)\phi(s) \, ds,\label{eq main 2}\\
	\phi(0)=&\gamma,\label{eq main 1IC2}
\end{align}
where $
a_1,a_2\in C[0,1]$; 
$
K_1,K_2\in C([0,1]\times[0,1]).
$ Define
\[
M_1
=
\max_{x\in[0,1]}
|a_1(x)|, 
M_2
=
\max_{x\in[0,1]}
|a_2(x)|,
M_3
=
\max_{(x,s)\in[0,1]^2}
|K_1(x,s)|,
M_4
=
\max_{(x,s)\in[0,1]^2}
|K_2(x,s)|.
\]
We now have the following estimates.
\begin{theorem}
	Suppose that $
	\phi\in H^{m}(0,1)$ 
$	(m=0,1,2,3,\cdots)$
and let $\phi_N(x)$ be the best approximation to $\phi$. 
Then for a sufficiently large $N$, we have
\begin{align}
	\left\| 	\phi(x)-	\phi_N(x)\right\|_{L^{2}} 
	\leq &C^{*} N^{\frac{1}{2}-m}\left\| \phi\right\|_{H^{m}_{w^{(\alpha,\beta)}}([0,1])} \nonumber\\
	&+C^{**}N^{-m}\left( \left\| K_{1}\right\|_{H^{m}_{w^{(\alpha,\beta)}}([0,1])}\left\|   \phi\right\|_{L^{2}}+ \left\| K_{2}\right\|_{H^{m}_{w^{(\alpha,\beta)}}([0,1])} \left\| \phi\right\|_{L^{2}}\right),\label{eq errmain}
\end{align}	
where $C^{*}$ and $C^{**}$ are independent of $N$.
	
\end{theorem}

\begin{proof}
Let $\phi_N(x)$ denote the approximation to the continuous function $\phi(x)$. Then
\begin{align}	
	\phi_N'(x)
	=&f(x)
	+a_1(x)I^{\alpha,\beta}_{N}\phi(x)
	+a_2(x)I^{\alpha,\beta}_{N}\phi(\mu x)\nonumber\\
	&
	+\int_0^xK_1(x,s)I^{\alpha,\beta}_{N}\phi(s)\,ds
	+\int_0^{\mu x}K_2(x,s)I^{\alpha,\beta}_{N}\phi(s)\,ds.
\end{align}
Thus
\begin{align}\label{eq phi-phiN}
\phi'(x)-	\phi_N'(x)
	=& a_{1}(x)\left( \phi(x)-I^{\alpha,\beta}_{N}\phi(x)\right) +a_{2}(x)\left( \phi( \mu x)-I^{\alpha,\beta}_{N}\phi(\mu x)\right) \nonumber\\
	&+\int_0^x K_{1}(x,s)\phi(s)\, ds-\int_0^xK_1(x,s)I^{\alpha,\beta}_{N}\phi(s)\,ds\nonumber\\
	&+\int_0^{\mu x} K_{2}(x,s)\phi(s) \, ds-\int_0^{\mu x}K_2(x,s)I^{\alpha,\beta}_{N}\phi(s)\,ds\nonumber\\
=& a_{1}(x)\left( \phi(x)-I^{\alpha,\beta}_{N}\phi(x)\right) +a_{2}(x)\left( \phi( \mu x)-I^{\alpha,\beta}_{N}\phi(\mu x)\right) \nonumber\\
&+\int_0^x K_{1}(x,s)\left( \phi(s)-I^{\alpha,\beta}_{N}\phi(s)\right)\, ds+\int_0^{\mu x} K_{2}(x,s)\left( \phi(s)-I^{\alpha,\beta}_{N}\phi(s)\right)\, ds\nonumber\\
&+\int_0^x\left( K_1(x,s)-I^{\alpha,\beta}_{N}K_1(x,s)\right) I^{\alpha,\beta}_{N}\phi(s)\,ds\nonumber\\
&+\int_0^{\mu x}\left( K_2(x,s)-I^{\alpha,\beta}_{N}K_2(x,s)\right) I^{\alpha,\beta}_{N}\phi(s)\,ds.
\end{align}
	Integrating \eqref{eq phi-phiN} from $0$ to $x$ and using the given initial condition, we have
\begin{align}
\phi(x)-	\phi_N(x)
=& \int_{0}^{x}a_{1}(t)\left( \phi(t)-I^{\alpha,\beta}_{N}\phi(t)\right)dt + \int_{0}^{x}a_{2}(t)\left( \phi( \mu t)-I^{\alpha,\beta}_{N}\phi(\mu t)\right)dt \nonumber\\
&+ \int_{0}^{x}\int_0^t K_{1}(t,s)\left( \phi(s)-I^{\alpha,\beta}_{N}\phi(s)\right)\, ds\,dt\nonumber\\
&+ \int_{0}^{x}\int_0^{\mu t} K_{2}(t,s)\left( \phi(s)-I^{\alpha,\beta}_{N}\phi(s)\right)\, ds\,dt\nonumber\\
&+ \int_{0}^{x}\int_0^t\left( K_1(t,s)-I^{\alpha,\beta}_{N}K_1(t,s)\right) I^{\alpha,\beta}_{N}\phi(s)\,ds\,dt\nonumber\\
&+ \int_{0}^{x}\int_0^{\mu t}\left( K_2(t,s)-I^{\alpha,\beta}_{N}K_2(t,s)\right) I^{\alpha,\beta}_{N}\phi(s)\,ds\,dt.
\end{align}	
Taking the $L^{2}$-norm, we have 
\begin{align}
\left\| 	\phi(x)-	\phi_N(x)\right\|_{L^{2}} 
\leq& \left\| \int_{0}^{x}a_{1}(t)\left( \phi(t)-I^{\alpha,\beta}_{N}\phi(t)\right)dt\right\|_{L^{2}}  +\left\|  \int_{0}^{x}a_{2}(t)\left( \phi( \mu t)-I^{\alpha,\beta}_{N}\phi(\mu t)\right)dt\right\|_{L^{2}}  \nonumber\\
&+\left\|  \int_{0}^{x}\int_0^t K_{1}(t,s)\left( \phi(s)-I^{\alpha,\beta}_{N}\phi(s)\right)\, ds\,dt\right\|_{L^{2}}\nonumber\\
&+\left\|  \int_{0}^{x}\int_0^{\mu t} K_{2}(t,s)\left( \phi(s)-I^{\alpha,\beta}_{N}\phi(s)\right)\, ds\,dt\right\|_{L^{2}}\nonumber\\
&+ \left\| \int_{0}^{x}\int_0^t\left( K_1(t,s)-I^{\alpha,\beta}_{N}K_1(t,s)\right) I^{\alpha,\beta}_{N}\phi(s)\,ds\,dt\right\|_{L^{2}}\nonumber\\
&+ \left\| \int_{0}^{x}\int_0^{\mu t}\left( K_2(t,s)-I^{\alpha,\beta}_{N}K_2(t,s)\right) I^{\alpha,\beta}_{N}\phi(s)\,ds\,dt\right\|_{L^{2}}.\label{eq err1}
\end{align}	
Using 
\begin{align}
	&\left\| \phi( \mu x)-I^{\alpha,\beta}_{N}\phi(\mu x)\right\|_{\infty} 	
	\le
	\left\| \phi( x)-I^{\alpha,\beta}_{N}\phi( x)\right\|_{\infty},\\
	&
	\int_0^{\mu x}\left\| F(x,s)\right\| ds
	\le
	\int_0^{x}\left\| F(x,s)\right\| ds,
	& 
\end{align}	
and Lemma \ref{le CSIneq} for the last two norms in \eqref{eq err1}, we have
\begin{align}
	\left\| 	\phi(x)-	\phi_N(x)\right\|_{L^{2}} 
	\leq &M_{1}\left\|  \phi(x)-I^{\alpha,\beta}_{N}\phi(x)\right\|_{\infty}  + M_{2}\left\| \phi( x)-I^{\alpha,\beta}_{N}\phi( x)\right\|_{\infty}  \nonumber\\
	&+M_{3}\left\| \phi(x)-I^{\alpha,\beta}_{N}\phi(x)\right\|_{\infty} +M_{4}\left\| \phi(x)-I^{\alpha,\beta}_{N}\phi(x)\right\|_{\infty}\nonumber\\
	&+ \left\|K_1(x,s)-I^{\alpha,\beta}_{N}K_1(x,s)\right\|_{L^{2}}\left\|   I^{\alpha,\beta}_{N}\phi(s)\right\|_{L^{2}}\nonumber\\
	&+ \left\| K_2(x,s)-I^{\alpha,\beta}_{N}K_2(x,s)\right\|_{L^{2}} \left\| I^{\alpha,\beta}_{N}\phi(s)\right\|_{L^{2}}.\label{eq err2}
\end{align}	
Upon applying Lemma \ref{le 1}, one has
\begin{align}
	\left\| 	\phi(x)-	\phi_N(x)\right\|_{L^{2}} 
	\leq &\left( C_{3} + C_{4}+C_{5}+C_{6}\right) N^{\frac{1}{2}-m}\left\| \phi\right\|_{H^{m}_{w^{(\alpha,\beta)}}([0,1])} \nonumber\\
	&+C_{7}N^{-m}\left( \left\| K_{1}\right\|_{H^{m}_{w^{(\alpha,\beta)}}([0,1])}\left\|   \phi\right\|_{L^{2}}+ \left\| K_{2}\right\|_{H^{m}_{w^{(\alpha,\beta)}}([0,1])} \left\| \phi\right\|_{L^{2}}\right) .\label{eq err3}
\end{align}

\end{proof}

\section{Numerical Examples}\label{Sec Exam}
This section presents illustrative examples of the pantograph delay integrodifferential equation of Volterra type \eqref{eq main 1}-\eqref{eq main 1IC}. The shifted Jacobi collocation algorithm developed in Section \ref{Sec Collo Ago} is implemented to give explicit numerical solutions of the proposed examples. Five collocation methods (based on different special values of $\alpha,\beta$) are considered, namely, collocation methods based on shifted Legendre polynomials and shifted Chebyshev polynomials of the first, second, third, and fourth kinds. The numerical solutions and errors obtained from these methods are compared with those presented in a very recent paper \cite{Aourir2025}.  Our collocation points are the zeros of shifted Jacobi polynomials presented in Tables \ref{table zeroP8912}-\ref{table zeroP1415}. We use Wolfram Mathematica software 12.0 throughout our computations.  

To this end, we define the absolute error by  $(0\leq i\leq 4)$ 
\begin{equation}
E_{N,i}(x)=\left|\phi_{{\rm exact}}(x)-\phi_{N,i}(x) \right|,
\end{equation}
while the $L_{\infty}$ (maximum absolute) error and the $L_{2}$-norm error are defined, respectively, by
\begin{equation}
	L_{\infty}:=\max\limits_{x\in(0,1]}\left|\phi_{{\rm exact}}(x)-\phi_{N,i}(x) \right| \quad  \mbox{and} \quad 	L_{2}:=\sqrt{\sum_{j}\left|\phi_{{\rm exact}}(x_{j})-\phi_{N,i}(x_{j}) \right|^{2}}. 
\end{equation}
Here $\phi_{{\rm exact}}(x)$ and $\phi_{N,i}(x)$ are the exact and approximate solutions, respectively. Specifically,  $\phi_{N,i}(x)$ $(i=0,1,2,3,4)$ are the numerical solutions obtained for $(\alpha,\beta)=(0,0)$, $(\alpha,\beta)=(-1/2,-1,2)$, $(\alpha,\beta)=(1/2,1/2)$, $(\alpha,\beta)=(-1/2,1/2)$, and  $(\alpha,\beta)=(1/2,-1/2)$, respectively.

\begin{example}\label{Example 5.1}
Consider the initial value problem for pantograph delay integrodifferential equation of Volterra type {\rm \cite{Aourir2025}, \cite{Behera2022},  \cite{22Jafari}}
\begin{align}
\phi'(x)=&\frac{1}{2}-\frac{x}{4}  e^{\frac{x}{4}}+\frac{x^2}{32}-\frac{1}{2} e^{3x}+e^{2 x}+\frac{1}{2} \phi(x)+\phi\left( \frac{x}{4}\right)\label{eq1 exam 5.1}\nonumber\\
&+\int_0^x e^{s+x} u(s)\, ds+\int_0^{\frac{x}{4}} s \,u(s) \, ds,\\
\phi(0)=&0.\label{eq1-1 exam 5.1}
\end{align}
The exact solution is {\rm \cite{Aourir2025}}
\begin{equation}
	\phi(x)=e^{x}-1.
\end{equation}

To solve the initial value problem \eqref{eq1 exam 5.1}-\eqref{eq1-1 exam 5.1}, we implement the shifted Jacobi collocation algorithm presented in Section \ref{Sec Collo Ago}. To this end, substituting the shifted Jacobi polynomial series solution \eqref{eqsoluv} into the integrodifferential equation in \eqref{eq1 exam 5.1}, we have
\begin{align}\label{eq2 exam 5.1}
&\sum _{n=1}^N c_n (\alpha +\beta +n+1) \mathsf{J}_{n-1}^{(\alpha +1,\beta +1)}(x)\nonumber\\
&=\frac{1}{2}-\frac{x}{4}  e^{\frac{x}{4}}+\frac{x^2}{32}-\frac{1}{2} e^{3x}+e^{2 x}+\frac{1}{2} \sum _{n=0}^N c_n \mathsf{J}_n^{(\alpha ,\beta )}(x)+\mathsf{J}_n^{(\alpha ,\beta )}\left(\frac{ x}{4}\right)\nonumber\\
	&+ \sum _{n=0}^Nc_n \int_0^x e^{s+x} \mathsf{J}_n^{(\alpha ,\beta )}(s)\, ds+\sum _{n=0}^N c_n \int_0^{\frac{x}{4}} s \,\mathsf{J}_n^{(\alpha ,\beta )}(s) \, ds.
\end{align}
Using the finite series representation \eqref{eq po ser1} for shifted Jacobi polynomials, 
we get
\begin{align}\label{eq3 exam 5.1}
	&\sum _{n=1}^N c_n (\alpha +\beta +n+1) \sum_{k=0}^{n-1}F_{n-1,k}^{\alpha+1,\beta+1}x^{k}\nonumber\\
	&=\frac{1}{2}-\frac{x}{4}  e^{\frac{x}{4}}+\frac{x^2}{32}-\frac{1}{2} e^{3x}+e^{2 x}+\frac{1}{2} \sum _{n=0}^N c_n \sum_{k=0}^{n}F_{n,k}^{\alpha,\beta}x^{k}+\sum _{n=0}^N c_n \sum_{k=0}^{n}4^{-k}F_{n,k}^{\alpha,\beta}x^{k}\nonumber\\
	&+ \sum _{n=0}^Nc_n \sum_{k=0}^{n}F_{n,k}^{\alpha,\beta}\int_0^x e^{s+x} s^{k}\, ds+\sum _{n=0}^N c_n \sum_{k=0}^{n}F_{n,k}^{\alpha,\beta} \int_0^{\frac{x}{4}} s^{k+1} \, ds.
\end{align}
Collocating at the zeros $ \mathsf{x}^{N-1}_{q}(\alpha,\beta)$ $(0\leq q\leq N-1)$ of the shifted Jacobi polynomials $\mathsf{J}^{(\alpha,\beta)}_{N}\left( x\right)$  so that the integrodifferential equation  \eqref{eq3 exam 5.1} is satisfied exactly at these collocation points. Thus we obtain a collocation scheme $(0\leq q\leq N-1)$
\begin{align}\label{eq4 exam 5.1}
	&\sum _{n=1}^N c_n (\alpha +\beta +n+1) \sum_{k=0}^{n-1}F_{n-1,k}^{\alpha+1,\beta+1}\left( \mathsf{x}^{N-1}_{q}(\alpha,\beta)\right) ^{k}\nonumber\\
	&=\frac{1}{2}-\frac{ \mathsf{x}^{N-1}_{q}(\alpha,\beta)}{4}  e^{\frac{ \mathsf{x}^{N-1}_{q}(\alpha,\beta)}{4}}+\frac{\left(  \mathsf{x}^{N-1}_{q}(\alpha,\beta)\right) ^2}{32}-\frac{1}{2} e^{3 \mathsf{x}^{N-1}_{q}(\alpha,\beta)}+e^{2  \mathsf{x}^{N-1}_{q}(\alpha,\beta)}\nonumber\\
	&+\sum _{n=0}^N c_n \sum_{k=0}^{n}F_{n,k}^{\alpha,\beta}\left(  \mathsf{x}^{N-1}_{q}(\alpha,\beta)\right) ^{k}+\frac{1}{2} \sum _{n=0}^N c_n \sum_{k=0}^{n}4^{-k}F_{n,k}^{\alpha,\beta}\left(  \mathsf{x}^{N-1}_{q}(\alpha,\beta)\right) ^{k}\nonumber\\
	&+ \sum _{n=0}^Nc_n \sum_{k=0}^{n}F_{n,k}^{\alpha,\beta}\int_0^{ \mathsf{x}^{N-1}_{q}(\alpha,\beta)} e^{s+ \mathsf{x}^{N-1}_{q}(\alpha,\beta)} s^{k}\, ds+\sum _{n=0}^N c_n \sum_{k=0}^{n}F_{n,k}^{\alpha,\beta} \int_0^{\frac{ \mathsf{x}^{N-1}_{q}(\alpha,\beta)}{4}} s^{k+1} \, ds.
\end{align}
The collocation scheme \eqref{eq4 exam 5.1} gives a system of $N$ algebraic equations for $N+1$ unknown expansion coefficients $c_{n}$ $(0\leq n\leq N)$. The remaining one equation comes from the initial condition
\begin{equation}\label{eq5 exam 5.1}
\sum_{k=0}^{N}c_{k}(-1)^{k}\frac{\Gamma(k+\beta+1)}{\Gamma(\beta+1)k!}=0.
\end{equation}
The collocation scheme \eqref{eq4 exam 5.1}-\eqref{eq5 exam 5.1} now gives a set of $N+1$ algebraic equations in the unknown coefficients $c_{n}$ $(0\leq n\leq N)$.

  We now evaluate the collocation scheme \eqref{eq4 exam 5.1}-\eqref{eq5 exam 5.1} explicitly by considering  special values $N=8,9$; and  $(\alpha,\beta)=(0,0)$, $(\alpha,\beta)=(-1/2,-1,2)$, $(\alpha,\beta)=(1/2,1/2)$, $(\alpha,\beta)=(-1/2,1/2)$, and  $(\alpha,\beta)=(1/2,-1/2)$. 

\paragraph{Shifted Legendre Collocation Method $(\alpha=\beta=0)$.} We proceed as follows:
\begin{enumerate}[$(a)$]
	\item  $N=8$. In this case, the zeros $\mathsf{x}^{7}_{q}(0,0)$ of the shifted Legendre polynomials $\mathsf{J}^{(0,0)}_{8}(x)=P_{8}(2x-1)$ are given in Table \ref{table zeroP8912}. Collocating the integral equation \eqref{eq4 exam 5.1} at these zeros, we obtain the following set of algebraic equations.
\begin{align}\label{eq6 exam 5.1}
&	2 c_1-5.76174 c_2+10.8323 c_3-16.5895 c_4+22.3036 c_5-27.2121 c_6+30.5974 c_7-31.8627 c_8\nonumber\\
&=1.52047 c_0-1.49028 c_1+1.43124 c_2-1.34598 c_3+1.2382 c_4-1.11252 c_5+0.974141 c_6\nonumber\\
&-0.828595 c_7+0.681393 c_8+1.00485,\nonumber\\
&2 c_1-4.78 c_2+6.52016 c_3-5.74692 c_4+2.15115 c_5+3.12605 c_6-7.90559 c_7+9.92334 c_8\nonumber\\
&=1.61879 c_0-1.45403 c_1+1.16191 c_2-0.805738 c_3+0.455747 c_4-0.169021 c_5\nonumber\\
&-0.0249294 c_6+0.12847 c_7-0.16922 c_8+1.02142,\nonumber\\
&2 c_1-3.15319 c_2+1.14276 c_3+2.80295 c_4-4.7428 c_5+2.11995 c_6+3.11938 c_7-5.93566 c_8\nonumber\\
&=1.84118 c_0-1.40149 c_1+0.752552 c_2-0.193292 c_3-0.104881 c_4+0.178561 c_5\nonumber\\
&-0.188293 c_6+0.255701 c_7-0.363178 c_8+1.02726,\nonumber\\
&2 c_1-1.10061 c_2-2.49528 c_3+2.53549 c_4+2.07263 c_5-3.87901 c_6-0.876384 c_7+4.77764 c_8\nonumber\\
&=2.26369 c_0-1.31987 c_1+0.273941 c_2+0.182909 c_3-0.178079 c_4+0.216539 c_5\nonumber\\
&-0.395283 c_6+0.380538 c_7-0.0331833 c_8+0.95306,\nonumber\\
&2 c_1+1.10061 c_2-2.49528 c_3-2.53549 c_4+2.07263 c_5+3.87901 c_6-0.876384 c_7-4.77764 c_8\nonumber\\
&=2.96942 c_0-1.13206 c_1-0.185774 c_2+0.193408 c_3-0.173211 c_4+0.474283 c_5\nonumber\\
&-0.273435 c_6-0.271355 c_7+0.354373 c_8+0.65441,\nonumber\\
&2 c_1+3.15319 c_2+1.14276 c_3-2.80295 c_4-4.7428 c_5-2.11995 c_6+3.11938 c_7+5.93566 c_8\nonumber\\
&=3.97157 c_0-0.71581 c_1-0.422452 c_2+0.00460963 c_3-0.485122 c_4+0.352849 c_5\nonumber\\
&+0.32181 c_6-0.221457 c_7+0.187275 c_8-0.0440563,\nonumber\\
&2 c_1+4.78 c_2+6.52016 c_3+5.74692 c_4+2.15115 c_5-3.12605 c_6-7.90559 c_7-9.92334 c_8\nonumber\\
&=5.09923 c_0-0.0591633 c_1-0.208579 c_2+0.131105 c_3-0.661985 c_4-0.214354 c_5\nonumber\\
&+0.110738 c_6-0.340864 c_7+0.155693 c_8-1.12916\nonumber\\
	&2 c_1+5.76174 c_2+10.8323 c_3+16.5895 c_4+22.3036 c_5+27.2121 c_6+30.5974 c_7+31.8627 c_8\nonumber\\
	&=5.96656 c_0+0.56115 c_1+0.278139 c_2+0.704686 c_3-0.110368 c_4+0.0760736 c_5\nonumber\\
	&+0.385415 c_6-0.244664 c_7-0.104238 c_8-2.1437,\nonumber\\
	&c_0-c_1+c_2-c_3+c_4-c_5+c_6-c_7+c_8=0.
\end{align}
Solving iteratively using Newton's method with starting points $c_i=0$ $(0\leq i\leq 8)$,
we obtain the solutions
\begin{align}\label{eq8 exam 5.1}
	\begin{split}
&	c_0= 0.718282,\quad c_1= 0.845155,\quad c_2= 0.139864,\quad c_3= 0.0139313,\quad c_4=0.000992588\\
	&c_5= 0.0000550476,\quad c_6= 2.49\times 10^{-6}, \quad c_7= 9.61\times 10^{-8},\quad c_8=3.20\times 10^{-9}.
	\end{split}
	\end{align}
Using the numerical values \eqref{eq8 exam 5.1} in the series solution \eqref{eqsoluv}, we obtain the polynomial solution
\begin{align}
\phi_{8,0}(x)=&9.18\times 10^{-17}+1.000000 x+0.500000 x^2+0.166665 x^3+0.0416736 x^4\nonumber\\
&+0.00831276 x^5+0.00142366 x^6+0.000165086 x^7+0.0000411942 x^8.
\end{align}

	\item $N=9$. Indeed, the zeros $\mathsf{x}^{8}_{q}(0,0)$ of the shifted Legendre polynomial $\mathsf{J}^{(0,0)}_{9}(x)=P_{9}(2x-1)$ are given in Table \ref{table zeroP8912}. Collocating at these zeros, we get a set of algebraic equations
\begin{align}\label{eq9 exam 5.1}
&2 c_1-5.80896 c_2+11.06 c_3-17.2397 c_4+23.7293 c_5-29.8547 c_6+34.9414 c_7-38.3707 c_8\nonumber\\
&+39.6326 c_9=1.51631 c_0-1.49217 c_1+1.44476 c_2-1.37575 c_3+1.28758 c_4-1.18329 c_5\nonumber\\
&+1.06641 c_6-0.940772 c_7+0.810355 c_8-0.67907 c_9+1.0039,\nonumber\\
&2 c_1-5.01619 c_2+7.48422 c_3-7.91151 c_4+5.52683 c_5-0.671517 c_6-5.23159 c_7\nonumber\\
&+10.1398 c_8-12.1285 c_9=1.59295 c_0-1.46226 c_1+1.22484 c_2-0.922655 c_3\nonumber\\
&+0.605458 c_4-0.319072 c_5+0.0954071 c_6+0.053207 c_7-0.134056 c_8+0.168766 c_9\nonumber\\
&+1.01805,\nonumber\\
&2 c_1-3.68023 c_2+2.64337 c_3+1.12378 c_4-4.85513 c_5+5.20305 c_6-1.25039 c_7\nonumber\\
&-4.30287 c_8+7.01511 c_9=1.75991 c_0-1.41824 c_1+0.881296 c_2-0.357521 c_3\nonumber\\
&+0.00321992 c_4+0.149225 c_5-0.179733 c_6+0.201381 c_7-0.273013 c_8+0.363495 c_9\nonumber\\
&+1.02949,\nonumber\\
&2 c_1-1.94552 c_2-1.4229 c_3+3.67058 c_4-0.899325 c_5-3.76315 c_6+3.69301 c_7\nonumber\\
&+1.73474 c_8-5.34994 c_9=2.0671 c_0-1.35879 c_1+0.467147 c_2+0.0754891 c_3\nonumber\\
&-0.190447 c_4+0.181318 c_5-0.286462 c_6+0.412197 c_7-0.323585 c_8-0.00297546 c_9\nonumber\\
&+0.999375,\nonumber\\
&2 c_1-2.90\times 10^{-12} c_2-3. c_3+7.26\times 10^{-12} c_4+3.75 c_5-1.27\times 10^{-11} c_6-4.375 c_7\nonumber\\
&+1.45\times 10^{-11} c_8+4.92187 c_9=2.57737 c_0-1.24691 c_1+0.0325593 c_2+0.232835 c_3\nonumber\\
&-0.149328 c_4+0.334238 c_5-0.444622 c_6+0.0888524 c_7+0.336123 c_8-0.347588 c_9\nonumber\\
&+0.843606,\nonumber\\
&2 c_1+1.94552 c_2-1.4229 c_3-3.67058 c_4-0.899325 c_5+3.76315 c_6+3.69301 c_7\nonumber\\
&-1.73474 c_8-5.34994 c_9=3.33417 c_0-1.00052 c_1-0.321703 c_2+0.118992 c_3\nonumber\\
&-0.25935 c_4+0.51873 c_5-0.0110905 c_6-0.384329 c_7+0.201588 c_8-0.142563 c_9\nonumber\\
&+0.433197,
\end{align}
\begin{align}
&2 c_1+3.68023 c_2+2.64337 c_3-1.12378 c_4-4.85513 c_5-5.20305 c_6-1.25039 c_7\nonumber\\
&+4.30287 c_8+7.01511 c_9=4.29957 c_0-0.543897 c_1-0.410591 c_2-0.0149795 c_3\nonumber\\
&-0.593966 c_4+0.180545 c_5+0.347571 c_6-0.143285 c_7+0.283939 c_8-0.00619507 c_9\nonumber\\
&-0.329943,	\nonumber\\
	&2 c_1+5.01619 c_2+7.48422 c_3+7.91151 c_4+5.52683 c_5+0.671517 c_6-5.23159 c_7\nonumber\\
	&-10.1398 c_8-12.1285 c_9=5.29362 c_0+0.0717492 c_1-0.123311 c_2+0.218103 c_3\nonumber\\
	&-0.608446 c_4-0.247734 c_5+0.0628161 c_6-0.431719 c_7+0.0139914 c_8+0.27356 c_9\nonumber\\
	&-1.34381,\nonumber\\
	&2 c_1+5.80896 c_2+11.06 c_3+17.2397 c_4+23.7293 c_5+29.8547 c_6+34.9414 c_7\nonumber\\
	&+38.3707 c_8+39.6326 c_9=6.01241 c_0+0.596444 c_1+0.311404 c_2+0.749068 c_3\nonumber\\
	&-0.0558017 c_4+0.131269 c_5+0.453239 c_6-0.170842 c_7-0.0392494 c_8+0.210571 c_9\nonumber\\
	&-2.20128,\nonumber\\
	&c_0-c_1+c_2-c_3+c_4-c_5+c_6-c_7+c_8-c_9=0.
\end{align}
Solving, using Newton's iterative method with the starting points $c_i=0$ $(0\leq i\leq 8)$,
we obtain the roots
\begin{align}\label{eq10 exam 5.1}
	&	c_0=0.718282,\quad c_1= 0.845155,\quad c_2= 0.139864,\quad c_3= 0.0139313,\quad c_4=0.000992588\nonumber\\
	&c_5=0.0000550476,\quad c_6= 2.49\times 10^{-6}, \quad c_7= 9.60\times 10^{-8},\quad c_8=3.20\times 10^{-9},\nonumber\\
	&c_9=9.40\times 10^{-11}.
\end{align}
Substituting the numerical values \eqref{eq10 exam 5.1} into the series solution \eqref{eqsoluv}, we obtain the approximate solution
\begin{align}
	\phi_{9,0}(x)=&2.24\times 10^{-17}+1.000000x+0.5 00000x^2+0.166667 x^3+0.0416663 x^4\nonumber\\
	&+0.00833493 x^5+0.00138523 x^6+0.000203511 x^7+0.0000206115 x^8\nonumber\\
	&+4.57\times 10^{-6} x^9.
\end{align}

\end{enumerate}

In the shifted Chebyshev collocation methods that follow, for brevity, we omit the details of the computation and present only the approximate solutions given as polynomials.  
  
\paragraph{Shifted Chebyshev Collocation Method of the First Kind $(\alpha=\beta=-1/2)$.} %
Here we have the approximate solutions
	\begin{align}
		\phi_{8,1}(x)=&1.90\times 10^{-16}+1.000000 x+0.500000 x^2+0.166666x^3+0.0416731 x^4\nonumber\\
		&+0.00831358x^5+0.00142311 x^6+0.000165197 x^7+0.0000412154 x^8.\\
		\phi_{9,1}(x)=&-5.71\times 10^{-18}+1.000000x+0.500000 x^2+0.166667 x^3+0.0416663 x^4\nonumber\\
		&+0.00833484 x^5+0.00138534 x^6+0.000203451 x^7+0.0000206221 x^8\nonumber\\
		&+4.57\times 10^{-6} x^9.
	\end{align}
	
%
%
%
%

\paragraph{Shifted Chebyshev Collocation Method of the Second Kind $(\alpha=\beta=1/2)$.}
 Indeed one has the approximating polynomials
 	\begin{align}
 	\phi_{8,2}(x)=&1.70\times 10^{-16}+1.000000x+0.500000 x^2+0.166665x^3+0.0416742 x^4\nonumber\\
 	&+0.00831203x^5+0.00142413 x^6+0.00016499 x^7+0.0000411755x^8\\
 	\phi_{9,2}(x)=&3.64\times 10^{-17}+1.000000 x+0.500000 x^2+0.166667 x^3+0.0416662 x^4\nonumber\\
 	&+0.00833501 x^5+0.00138513 x^6+0.000203564 x^7+0.0000206021 x^8\nonumber\\
 	&+4.57\times 10^{-6} x^9.
 \end{align}

\paragraph{Shifted Chebyshev Collocation Method of the Third Kind $(\alpha=-1/2,\beta=1/2)$.} It is obtained in this case the approximate solutions 
\begin{align}
	\phi_{8,3}(x)=&2.22\times 10^{-16}+1.000000x+0.500000 x^2+0.166664 x^3+0.0416776 x^4\nonumber\\
	&+0.00830476 x^5+0.00143273 x^6+0.000159699 x^7+0.0000424979 x^8\\
	\phi_{9,3}(x)=&5.44\times 10^{-17}+1. 000000x+0.500000 x^2+0.166667 x^3+0.041666 x^4\nonumber\\
	&+0.00833567x^5+0.00138402 x^6+0.000204663 x^7+0.0000200162 x^8\nonumber\\
	&+4.70\times 10^{-6} x^9.
\end{align}

\paragraph{Shifted Chebyshev Collocation Method of the Fourth Kind $(\alpha=1/2,\beta=-1/2)$.}
Here we obtain the approximate solutions
\begin{align}
	\phi_{8,4}(x)=&1.000000 x+0.500000 x^2+0.166666 x^3+0.0416709 x^4+0.0083189 x^5\nonumber\\
	&+0.00141602 x^6+0.000169981 x^7+0.0000399279 x^8\\
	\phi_{9,4}(x)=&1.72\times 10^{-16}+1.000000 x+0.500000 x^2+0.166667 x^3+0.0416664 x^4\nonumber\\
	&+0.00833439 x^5+0.00138619 x^6+0.000202521 x^7+0.0000211583 x^8\nonumber\\
	&+4.44\times 10^{-6} x^9.
\end{align}

\end{example}

\begin{example}\label{Example 5.2}
Given the initial value problem for the pantograph delay integrodifferential equation of the Volterra type {\rm \cite{Aourir2025}, \cite{Behera2022}, \cite{22Jafari}}
	\begin{align}\label{eq1 exam 5.2}
		\phi'(x)=&\frac{1}{x+1}-\ln (1+x) \left(\frac{x}{2} \ln (1+x)+1\right) +\phi(x)-\frac{1}{2} \ln \left(1+\frac{x}{2}\right)\phi\left(\frac{x}{2} \right)\nonumber\\
		&+\int_0^x \frac{x}{s+1}u(s)\, ds+\int_0^{\frac{x}{2}} \frac{1}{s+1}u(s) \, ds,\\
		\phi(0)=&0.
	\end{align}
	The exact solution is \cite{Aourir2025}
	\begin{equation}
		u(x)=\ln(1+x).
	\end{equation}
Here, the collocation scheme gives $(0\leq q\leq N-1)$
	\begin{align}\label{eq4 exam 5.2}
		&\sum _{n=1}^N c_n (\alpha +\beta +n+1) \sum_{k=0}^{n-1}F_{n-1,k}^{\alpha+1,\beta+1}\left( \mathsf{x}^{N-1}_{q}(\alpha,\beta)\right) ^{k}\nonumber\\
		&=\frac{1}{\mathsf{x}^{N-1}_{q}(\alpha,\beta)+1}-\ln \left( 1+\mathsf{x}^{N-1}_{q}(\alpha,\beta)\right)  \left(\frac{\mathsf{x}^{N-1}_{q}(\alpha,\beta)}{2} \ln \left( 1+\mathsf{x}^{N-1}_{q}(\alpha,\beta)\right) +1\right)\nonumber\\
		&+\sum _{n=0}^N c_n \sum_{k=0}^{n}F_{n,k}^{\alpha,\beta}\left(  \mathsf{x}^{N-1}_{q}(\alpha,\beta)\right) ^{k}-\frac{1}{2}\ln \left( 1+\frac{\mathsf{x}^{N-1}_{q}(\alpha,\beta)}{2}\right) \sum _{n=0}^N c_n \sum_{k=0}^{n}2^{-k}F_{n,k}^{\alpha,\beta}\left(  \mathsf{x}^{N-1}_{q}(\alpha,\beta)\right) ^{k}\nonumber\\
		&+ \sum _{n=0}^Nc_n \sum_{k=0}^{n}F_{n,k}^{\alpha,\beta}\int_0^{ \mathsf{x}^{N-1}_{q}(\alpha,\beta)} \frac{\mathsf{x}^{N-1}_{q}(\alpha,\beta)}{1+s} s^{k}\, ds+\sum _{n=0}^N c_n \sum_{k=0}^{n}F_{n,k}^{\alpha,\beta} \int_0^{\frac{ \mathsf{x}^{N-1}_{q}(\alpha,\beta)}{2}} \frac{s^{k}}{s+1} \, ds,
	\end{align}
satisfying the initial condition
	\begin{equation}\label{eq5 exam 5.2}
	\sum_{k=0}^{N}c_{k}(-1)^{k}\frac{\Gamma(k+\beta+1)}{\Gamma(\beta+1)k!}=0.
	\end{equation}
	The collocation scheme \eqref{eq4 exam 5.2}-\eqref{eq5 exam 5.2} now gives a set of $N+1$ algebraic equations in the unknown coefficients $c_{n}$ $(0\leq n\leq N)$.

	We now evaluate the collocation scheme \eqref{eq4 exam 5.2}-\eqref{eq5 exam 5.2} explicitly by considering  special values $N=14,15$; and  $(\alpha,\beta)=(0,0)$, $(\alpha,\beta)=(-1/2,-1,2)$, $(\alpha,\beta)=(1/2,1/2)$, $(\alpha,\beta)=(-1/2,1/2)$, and  $(\alpha,\beta)=(1/2,-1/2)$.
	
	\paragraph{Legendre Collocation Method $(\alpha=\beta=0)$.} We have the approximate solutions
	\begin{align}
		\phi_{14,0}(x)=&-2.08\times 10^{-17}+1.000000x-0.500000 x^2+0.333333 x^3-0.249992x^4\nonumber\\
		&+0.199915 x^5-0.166101x^6+0.140259 x^7-0.116398 x^8+0.0899575 x^9\nonumber\\
		&-0.060288x^{10}+0.0322921 x^{11}-0.0125706 x^{12}+0.00310052 x^{13}\nonumber\\
		&-0.000359495 x^{14}\\
		\phi_{15,0}(x)=&3.60\times 10^{-17}+1. 000000x-0.500000 t^2+0.333333 x^3-0.249998 t^4+0.199974 x^5\nonumber\\
		&-0.16647 t^6+0.141812 x^7-0.120995 x^8+0.0996671 x^9-0.0750085 x^{10}\nonumber\\
		&+0.0481863 x^{11}-0.0244988 x^{12}+0.00901339 x^{13}-0.00209998 x^{14}\nonumber\\
		&+0.000230368 x^{15}.
	\end{align}

	\paragraph{Shifted Chebyshev Collocation Method of the First Kind $(\alpha=\beta=-1/2)$.} 
	We have the approximate solutions
	\begin{align}
		\phi_{14,1}(x)=&7.65\times 10^{-18}+1.000000 x-0.500000 x^2+0.333333 x^3-0.249993 x^4+0.199927 x^5\nonumber\\
		&-0.166156 x^6+0.140427 x^7-0.116758 x^8+0.0904952 x^9-0.0608472 x^{10}\nonumber\\
		&+0.0326877 x^{11}-0.0127513 x^{12}+0.00314837 x^{13}-0.000365028 x^{14}\\
		\phi_{15,1}(x)=&-8.01\times 10^{-17}+1.000000 x-0.500000 t^2+0.333333x^3-0.249998 x^4\nonumber\\
		&+0.199978x^5-0.166491 x^6+0.141893x^7-0.121203 x^8+0.100052 x^9\nonumber\\
		&-0.0755168 x^{10}+0.0486628x^{11}-0.0248072 x^{12}+0.00914394 x^{13}\nonumber\\
		&-0.00213236 x^{14}+0.000233912 x^{15}.
	\end{align}

	\paragraph{Shifted Chebyshev Collocation Method of the Second Kind $(\alpha=\beta=1/2)$.} 
Indeed one has the approximating polynomials
	\begin{align}
		\phi_{14,2}(x)=&9.06\times 10^{-17}+1.000000x-0.500000 x^2+0.333333x^3-0.24999 x^4+0.199903x^5\nonumber\\
		&-0.166048 x^6+0.140098 x^7-0.11606 x^8+0.0894585 x^9-0.059773 x^{10}\nonumber\\
		&+0.0319295x^{11}-0.0124053 x^{12}+0.0030568 x^{13}-0.00035444 x^{14}\\
		\phi_{15,2}(x)=&7.85\times 10^{-17}+1.000000 x-0.5 x^2+0.333333 x^3-0.249997 x^4+0.19997 x^5\nonumber\\
		&	-0.166448 x^6+0.141734 x^7-0.120797 x^8+0.0993061 x^9-0.0745355 x^{10}\nonumber\\
		&+0.0477456x^{11}-0.0242148 x^{12}+0.00889341 x^{13}-0.00207025 x^{14}\nonumber\\
		&+0.000227114x^{15}.
	\end{align}

	\paragraph{Shifted Chebyshev Collocation Method of the Third Kind $(\alpha=-1/2,\beta=1/2)$.} 	
We obtain the approximating polynomials
	\begin{align}
		\phi_{14,3}(x)=&5.55\times 10^{-17}+1.000000x-0.5 x^2+0.333332 x^3-0.249986 x^4+0.199873x^5\nonumber\\
		&-0.165895 x^6+0.139568x^7-0.114782 x^8+0.0872792x^9-0.0571644 x^{10}\nonumber\\
		&+0.0297834x^{11}-0.0112498 x^{12}+0.00269018 x^{13}-0.000302477x^{14}\\
		\phi_{15,3}(x)=&-1.45\times 10^{-16}+1.000000 x-0.500000 x^2+0.333333 x^3-0.249996 x^4\nonumber\\
		&+0.19996 x^5-0.166389 x^6+0.141495 x^7-0.120103 x^8+0.0978629x^9\nonumber\\
		&-0.0723719 x^{10}+0.0454275x^{11}-0.0224838 x^{12}+0.00803788 x^{13}\nonumber\\
		&-0.00181874 x^{14}+0.000193821 x^{15}.
	\end{align}

	\paragraph{Shifted Chebyshev Collocation Method of the Fourth Kind $(\alpha=1/2,\beta=-1/2)$.} 
In this case, the collocation scheme \eqref{eq4 exam 5.2} yields 
\begin{align}
	\phi_{14,4}(x)=&2.77\times 10^{-17}+1.000000 x-0.500000 x^2+0.333333 x^3-0.249995 x^4\nonumber\\
	&+0.199945x^5-0.166264 x^6+0.140844 x^7-0.117865 x^8+0.0925331 x^9\nonumber\\
	&-0.0634533 x^{10}+0.0349573x^{11}-0.0140355 x^{12}+0.00357403 x^{13}\nonumber\\
	&-0.000427748 x^{14}\\
	\phi_{15,4}(x)=&5.55\times 10^{-17}+1.000000x-0.500000 x^2+0.333333 x^3-0.249999 x^4\nonumber\\\
	&+0.199984 x^5-0.166531 x^6+0.142073 x^7-0.121776 x^8+0.101343 x^9\nonumber\\\
	&-0.0775887 x^{10}+0.0510169x^{11}-0.0266575 x^{12}+0.0101007 x^{13}\nonumber\\\
	&-0.00242519x^{14}+0.000274097 x^{15}.
\end{align}

\end{example}

\begin{example}\label{Example 5.3}
	Consider the initial value problem for the pantograph delay integrodifferential equation of Volterra type {\rm \cite{Ahmad2021}, \cite{Aourir2025}}
\begin{align}\label{eq1 exam 5.3}
	\phi'(x)=&2+2x-u(x)-u\left( \frac{x}{2}\right) +x (2 x+1) \int_0^x e^{s (x-s)} u(s) \, ds\nonumber\\
	&+\frac{1}{2} \int_0^{\frac{x}{2}} x e^{s  \left(\frac{x}{2}-s\right)} u(s) \, ds,\\
	\phi(0)=&1.
\end{align}
The exact solution is \cite{Aourir2025}
\begin{equation}
	u(x)=e^{x^{2}}.
\end{equation}
Indeed the collocation scheme yields $(0\leq q\leq N-1)$
\begin{align}\label{eq4 exam 5.3}
	&\sum _{n=1}^N c_n (\alpha +\beta +n+1) \sum_{k=0}^{n-1}F_{n-1,k}^{\alpha+1,\beta+1}\left( \mathsf{x}^{N-1}_{q}(\alpha,\beta)\right) ^{k}\nonumber\\
	&=2+2\mathsf{x}^{N-1}_{q}(\alpha,\beta)-\sum _{n=0}^N c_n \sum_{k=0}^{n}F_{n,k}^{\alpha,\beta}\left(  \mathsf{x}^{N-1}_{q}(\alpha,\beta)\right) ^{k}-\sum _{n=0}^N c_n \sum_{k=0}^{n}2^{-k}F_{n,k}^{\alpha,\beta}\left(  \mathsf{x}^{N-1}_{q}(\alpha,\beta)\right) ^{k}\nonumber\\
	&+\mathsf{x}^{N-1}_{q}(\alpha,\beta)\left(2\mathsf{x}^{N-1}_{q}(\alpha,\beta)+1 \right)\sum _{n=0}^Nc_n \sum_{k=0}^{n}F_{n,k}^{\alpha,\beta}\int_0^{ \mathsf{x}^{N-1}_{q}(\alpha,\beta)}e^{s\left( \mathsf{x}^{N-1}_{q}(\alpha,\beta)-s\right) } s^{k}\, ds\nonumber\\
	&+\frac{1}{2}\sum _{n=0}^N c_n \sum_{k=0}^{n}F_{n,k}^{\alpha,\beta} \int_0^{\frac{ \mathsf{x}^{N-1}_{q}(\alpha,\beta)}{2}}\mathsf{x}^{N-1}_{q}(\alpha,\beta) e^{s\left( \frac{\mathsf{x}^{N-1}_{q}(\alpha,\beta)}{2}-s\right) } s^{k} \, ds,
\end{align}
satisfying the initial condition
\begin{equation}\label{eq5 exam 5.3}
	\sum_{k=0}^{N}c_{k}(-1)^{k}\frac{\Gamma(k+\beta+1)}{\Gamma(\beta+1)k!}=1.
\end{equation}
The collocation scheme \eqref{eq4 exam 5.3}-\eqref{eq5 exam 5.3} now gives a set of $N+1$ algebraic equations in the unknown coefficients $c_{n}$ $(0\leq n\leq N)$.

We now evaluate the collocation scheme \eqref{eq4 exam 5.3}-\eqref{eq5 exam 5.3} explicitly by considering  special values $N=12,14$; and  $(\alpha,\beta)=(0,0)$, $(\alpha,\beta)=(-1/2,-1,2)$, $(\alpha,\beta)=(1/2,1/2)$, $(\alpha,\beta)=(-1/2,1/2)$, and  $(\alpha,\beta)=(1/2,-1/2)$.

\paragraph{Legendre Collocation Method $(\alpha=\beta=0)$.} 

In this case, the collocation scheme \eqref{eq4 exam 5.3}-\eqref{eq5 exam 5.3} gives
\begin{align}
	\phi_{12,0}(x)=&1.000000-1.74\times 10^{-8}x+1.000000 x^2-0.0000351514 x^3+0.500436 x^4\nonumber\\
	&-0.00308718 x^5+0.180335 x^6-0.039637x^7+0.118535 x^8-0.0994707x^9\nonumber\\
	&+0.0916269 x^{10}-0.0415608 x^{11}+0.0111388 x^{12}\\
	\phi_{14,0}(x)=&1.000000-2.23\times 10^{-10}x+1.000000 x^2-8.16\times 10^{-7} x^3+0.500014x^4\nonumber\\
	&-0.000135272 x^5+0.167511x^6-0.00354297 x^7+0.0519986x^8-0.0212634 x^9\nonumber\\
	&+0.0392504 x^{10}-0.0312716 x^{11}+0.0225265 x^{12}-0.00872902 x^{13}\nonumber\\
	&+0.00192493 x^{14}.
\end{align}

	\paragraph{Shifted Chebyshev Collocation Method of the First Kind $(\alpha=\beta=-1/2)$.} 
From the collocation scheme \eqref{eq4 exam 5.3}-\eqref{eq5 exam 5.3}, we have the approximate solutions
\begin{align}
	\phi_{12,1}(x)=&1.000000-5.66\times 10^{-9} x+1.000000 x^2-0.0000260323 x^3+0.500362 x^4\nonumber\\
	&-0.00274959 x^5+0.179405 x^6-0.0380706 x^7+0.117004 x^8-0.0987928 x^9\nonumber\\
	&+0.0917435 x^{10}-0.0418047 x^{11}+0.0112102 x^{12}\\
	\phi_{14,1}(x)=&1.000000-6.71\times 10^{-11} x+1.000000 x^2-5.71\times 10^{-7} x^3+0.500011 x^4\nonumber\\
	&-0.000115694 x^5+0.167427x^6-0.00330822 x^7+0.0515634x^8-0.0207404 x^9\nonumber\\
	&+0.0388765 x^{10}-0.0311625 x^{11}+0.0225697x^{12}-0.00877353 x^{13}\nonumber\\
	&+0.00193539 x^{14}.
\end{align}

	\paragraph{Shifted Chebyshev Collocation Method of the Second Kind $(\alpha=\beta=1/2)$.} 
	
We obtain from the collocation scheme \eqref{eq4 exam 5.3}-\eqref{eq5 exam 5.3}, the approximate solutions
	\begin{align}
		\phi_{12,2}(x)=&1.000000-3.64\times 10^{-8} x+1.000000 x^2-0.0000451135 x^3+0.50051 x^4\nonumber\\
		&-0.00341464 x^5+0.181212 x^6-0.0410844x^7+0.119931 x^8-0.100082x^9\nonumber\\
		&+0.0915176 x^{10}-0.0413383 x^{11}+0.0110738 x^{12}\\
		\phi_{14,2}(x)=&1.000000-4.98\times 10^{-10}x+1.000000 x^2-1.10\times 10^{-6} x^3+0.500017x^4\nonumber\\
		&-0.000155224 x^5+0.167593 x^6-0.00376807 x^7+0.0524087x^8-0.0217504 x^9\nonumber\\
		&+0.0395956 x^{10}-0.0313713 x^{11}+0.0224861 x^{12}-0.00868787 x^{13}\nonumber\\
		&+0.00191527 x^{14}.
	\end{align}

	\paragraph{Shifted Chebyshev Collocation Method of the Third Kind $(\alpha=-1/2,\beta=1/2)$.} 	
Using the given zeros in the collocation scheme \eqref{eq4 exam 5.3}, we obtain the approximate solutions
	\begin{align}
		\phi_{12,3}(x)=&1.000000-7.86\times 10^{-8} x+1.000000 x^2-0.0000845049 x^3+0.500896 x^4\nonumber\\
		&-0.00563572 x^5+0.189315 x^6-0.0605217 x^7+0.151004 x^8-0.132864 x^9\nonumber\\
		&+0.113433 x^{10}-0.0497489 x^{11}+0.0124851 x^{12}\\
		\phi_{14,3}(x)=&1.000000-1.07\times 10^{-9}x+1.000000 x^2-2.09\times 10^{-6} x^3+0.50003 x^4\nonumber\\
		&-0.000263504 x^5+0.16816x^6-0.00577882 x^7+0.0573763x^8-0.030388 x^9\nonumber\\
		&+0.050126x^{10}-0.0401812 x^{11}+0.0273032 x^{12}-0.010238 x^{13}\nonumber\\
		&+0.00213788 x^{14}.
	\end{align}

\paragraph{Shifted Chebyshev Collocation Method of the Fourth Kind $(\alpha=1/2,\beta=-1/2)$.} 
The collocation scheme \eqref{eq4 exam 5.3}-\eqref{eq5 exam 5.3} yields the approximate solutions
\begin{align}
	\phi_{12,4}(x)=&1.000000-2.52\times 10^{-9} x+1.000000 x^2-0.0000135026x^3+0.500201 x^4\nonumber\\
	&-0.00163182x^5+0.174707 x^6-0.0254794 x^7+0.0949818 x^8-0.0737596x^9\nonumber\\
	&+0.0739193 x^{10}-0.0345841 x^{11}+0.00994032 x^{12}\\
	\phi_{14,4}(x)=&1.000000-3.01\times 10^{-11} x+1.000000 x^2-2.93\times 10^{-7} x^3+0.500006 x^4\nonumber\\
	&-0.0000668341 x^5+0.167131 x^6-0.00212646 x^7+0.0483529 x^8-0.0146977 x^9\nonumber\\
	&+0.030997 x^{10}-0.024177 x^{11}+0.0185522 x^{12}-0.00742186 x^{13}\nonumber\\
	&+0.00173346 x^{14}.
\end{align}

\end{example}	

\section{Results and Discussion}	
In this section, we elaborate on the results presented in Section \ref{Sec Exam}. Three illustrative examples of the governing pantograph delay integrodifferential equation are considered, namely, Examples \ref{Example 5.1}-\ref{Example 5.3}. Our collocation points are the zeros of the underlying shifted Jacobi polynomials. All the zeros encountered in our computations are listed in Tables \ref{table zeroP8912}-\ref{table zeroP1415}. The results obtained in the three examples are compared with those in \cite{Aourir2025} and are presented in Tables \ref{table Example 5.1 num sol}-\ref{table Example5.3sol&errorsN=14}. It is to be noted that our choices of order $N$ of the approximation are for comparison purposes and in line with those considered in the paper \cite{Aourir2025} under comparison. Similarly, the chosen  parameters $\alpha,\beta$ are for identification according to whether the underlying basis functions are shifted Legendre polynomials, shifted Chebyshev polynomials of the first kind, shifted Chebyshev polynomials of the second kind, shifted Chebyshev polynomials of the third kind, or shifted Chebyshev polynomials of the fourth kind. Thus in each example, we apply shifted Legendre collocation method, shifted Chebyshev collocation method of the first kind,  shifted Chebyshev collocation method of the second kind,  shifted Chebyshev collocation method of the third kind, and shifted Chebyshev collocation method of the fourth kind. Similar computations can be done for several other values of $\alpha,\beta\in(-1,\infty)$. We now present the discussion of results obtained in the three examples one after the other.

Example \ref{Example 5.1} ($N=8,9$): Table \ref{table Example 5.1 num sol} gives numerical values of approximate solutions obtained using the present methods, and their comparison with approximate solutions reported in \cite[$N=8$ in Example 6.1]{Aourir2025}; $x=0.1,0.2,\cdots, 1.0$. In Table \ref{table Example5.1sol&errorsN=8}, we give absolute errors, $L_{2}$-norm errors, and  $L_{\infty}$-norm errors obtained from our methods and their comparison with their counterparts presented in \cite[Example 6.1]{Aourir2025}. Of all the methods considered in Table \ref{table Example5.1sol&errorsN=8}, the shifted Legendre collocation method gives the smallest $L_{2}$-norm and $L_{\infty}$-norm errors. Four of our five methods give better $L_{2}$-norm error compared to the one presented in \cite[Example 6.1]{Aourir2025}.  Moreover, in view of $L_{\infty}$-norm errors generated by our respective methods, as shown in Tables \ref{table Example5.1sol&errorsN=8} and \ref{table Example5.1sol&errorsN=9}, it is clear that the present methods yield better accurate results compared to the existing results in \cite[$N=8,9$ in Example 6.1]{Aourir2025}.

Example \ref{Example 5.2} ($N=14,15$): In Table \ref{table Example 5.2 num sol}, we present  numerical values of approximate solutions obtained using the present methods, and their comparison with approximate solutions reported in \cite[$N=14$  in Example 6.2]{Aourir2025} for $x=0.1,0.2,\cdots, 1.0$. Table \ref{table Example5.2sol&errorsN=14} gives absolute errors, $L_{2}$-norm errors, and  $L_{\infty}$-norm errors obtained from our methods and their comparison with those given in \cite[Example 6.2]{Aourir2025}. In this example as well, the shifted Legendre collocation method gives the smallest $L_{2}$-norm and $L_{\infty}$-norm errors, as shown in Table \ref{table Example5.2sol&errorsN=14}. All the present methods yield better $L_{2}$-norm and $L_{2}$-norm errors compared to the ones given in \cite[Example 6.2]{Aourir2025}.  Furthermore, in view of $L_{\infty}$-norm errors generated by our respective methods, as shown in Tables \ref{table Example5.2sol&errorsN=14} and \ref{table Example5.2sol&errorsN=15}, it is revealed that the present methods yield better accurate results compared to the existing results in \cite[$N=14,15$ in Example 6.2]{Aourir2025}.

Example \ref{Example 5.3} ($N=12,14$): Table \ref{table Example 5.3 num sol} presents  numerical values of approximate solutions obtained using the present methods, and their comparison with approximate solutions reported in \cite[$N=12$ in Example 6.3]{Aourir2025} for $x=0.1,0.2,\cdots, 1.0$. In Table \ref{table Example5.3sol&errorsN=12}, we give absolute errors, $L_{2}$-norm errors, and  $L_{\infty}$-norm errors obtained from our methods and the comparison of $L_{\infty}$-norm errors with those given in \cite[Example 6.3]{Aourir2025}. It is observed that the shifted Legendre collocation method gives the smallest  $L_{\infty}$-norm errors, as shown in Table \ref{table Example5.3sol&errorsN=12}, while the shifted Chebyshev collocation method of the second kind gives the smallest $L_{2}$-norm error, as demonstrated in Table \ref{table Example5.3sol&errorsN=12}. In Table \ref{table Example5.3sol&errorsN=14}, all the present methods yield better $L_{2}$-norm and $L_{\infty}$-norm errors compared to the ones given in \cite[Example 6.3]{Aourir2025}.  Hence, the present methods yield more accurate results compared to the existing results in \cite[$N=12,14$ in Example 6.3]{Aourir2025}.

\begin{table}[H]
\caption {Comparison of numerical solutions $\phi_{8,i}(x)$ $(i=0,1,2,3,4)$ with the exact solution, and other published results for Example \ref{Example 5.1}.}
\centering 
\begin{tabular}{|c| c| c |c|c|c|c|c|} 
	\hline
	$x$&$\phi_{{\rm ex.}}(x)$&$\phi_{8,0}(x)$&$\phi_{8,1}(x)$&$\phi_{8,2}(x)$&$\phi_{8,3}(x)$&$\phi_{8,4}(x)$&BCM \cite{Aourir2025}
	\\[0.5ex] 
	\hline 
	$0.1$ &$0.105171$&$0.105171$&$0.105171$&$0.105171$&$0.105171$&$0.105171$&$0.105171$
	\\
	\hline
	$0.2$ &$0.221403$&$ 0.221403$&$0.221403$&$ 0.221403$&$0.221403$&$ 0.221403$& $0.221403$
	\\
	\hline
	$0.3$ &$0.349859$&$ 0.349859$&$0.349859$&$ 0.349859$&$0.349859$&$ 0.349859$&$0.349859$
	\\
	\hline
	$0.4$ &$0.491825$&$ 0.491825$&$0.491825$&$ 0.491825$&$0.491825$&$ 0.491825$& $0.491825$
	\\
	\hline
	$0.5$ &$0.648721$&$ 0.648721$&$0.648721$&$ 0.648721$&$0.648721$&$ 0.648721$& $0.648721$
	\\
	\hline
	$0.6$ &$0.822119$&$ 0.822119$&$0.822119$&$ 0.822119$&$0.822119$&$ 0.822119$& $0.822119$
	\\
	\hline
	$0.7$ &$1.013750$&$ 1.013750$&$1.013750$&$ 1.013750$&$1.013750$&$ 1.013750$&$1.013750$
	\\
	\hline
	$0.8$ &$1.225540$&$ 1.225540$&$1.225540$&$ 1.225540$&$1.225540$&$ 1.225540$&$1.225540$
	\\
	\hline
	$0.9$ &$1.459600$&$ 1.459600$&$1.459600$&$ 1.459600$&$1.459600$&$ 1.459600$ & $1.459600$
	\\
	\hline
	$1.0$ &$1.718280$&$ 1.718280$&$1.718280$&$ 1.718280$&$1.718280$&$ 1.718280$& $1.718280$
	\\
	\hline
\end{tabular}
\label{table Example 5.1 num sol} 
\end{table}

\begin{table}[H]
\caption {Comparison of numerical errors $E_{8,i}(x)$ $(i=0,1,2,3,4)$ with existing errors for Example \ref{Example 5.1}.}
\centering 
\begin{tabular}{|c| c| c |c|c|c|c|} 
	\hline
	$x$&$E_{8,0}(x)$&$E_{8,1}(x)$&$E_{8,2}(x)$&$E_{8,3}(x)$&$E_{8,4}(x)$&BCM \cite{Aourir2025}
	\\[0.5ex] 
	\hline 
	$0.1$ &$3.74\times 10^{-11}$&$5.00\times 10^{-11}$&$1.05\times 10^{-11}$&$4.11\times 10^{-11}$&$1.82\times 10^{-11}$& $5.39\times 10^{-12}$
	\\
	\hline
	$0.2$ &$3.73\times 10^{-11}$&$4.63\times 10^{-11}$&$6.57\times 10^{-11}$&$1.07\times 10^{-10}$&$2.61\times 10^{-11}$& $1.08\times 10^{-11}$
	\\
	\hline
	$0.3$ &$2.01\times 10^{-11}$&$9.29\times 10^{-12}$&$7.79\times 10^{-11}$&$1.89\times 10^{-10}$&$2.69\times 10^{-11}$&$2.31\times 10^{-11}$
	\\
	\hline
	$0.4$ &$5.06\times 10^{-11}$&$1.03\times 10^{-10}$&$3.33\times 10^{-11}$&$1.17\times 10^{-10}$&$4.65\times 10^{-11}$&$9.25\times 10^{-12}$  
	\\
	\hline
	$0.5$ &$5.31\times 10^{-12}$&$3.51\times 10^{-11}$&$7.51\times 10^{-11}$&$1.30\times 10^{-10}$&$2.31\times 10^{-11}$&$2.73\times 10^{-11}$ 
	\\
	\hline
	$0.6$ &$4.48\times 10^{-11}$&$3.94\times 10^{-11}$&$1.23\times 10^{-10}$&$2.35\times 10^{-10}$&$1.72\times 10^{-11}$&$1.92\times 10^{-11}$ 
	\\
	\hline
	$0.7$ &$2.16\times 10^{-11}$&$4.90\times 10^{-11}$&$9.08\times 10^{-11}$&$2.63\times 10^{-10}$&$7.21\times 10^{-11}$&$4.66\times 10^{-11}$
	\\
	\hline
	$0.8$ &$3.99\times 10^{-11}$&$1.10\times 10^{-10}$&$1.20\times 10^{-10}$&$2.68\times 10^{-10}$&$1.88\times 10^{-11}$&$5.72\times 10^{-11}$ 
	\\
	\hline
	$0.9$ &$4.03\times 10^{-11}$&$1.60\times 10^{-11}$&$2.11\times 10^{-10}$&$4.06\times 10^{-10}$&$2.75\times 10^{-11}$&$7.12\times 10^{-11}$ 
	\\
	\hline
	$1.0$ &$3.75\times 10^{-12}$&$7.95\times 10^{-11}$&$2.75\times 10^{-10}$&$5.16\times 10^{-10}$&$4.87\times 10^{-11}$&$8.27\times 10^{-10}$ 
	\\
	\hline\hline
	$L_{\infty}$ &$5.06\times 10^{-11}$&$1.10\times 10^{-10}$&$2.75\times 10^{-10}$&$5.16\times 10^{-10}$&$7.21\times 10^{-11}$& $8.27\times 10^{-10}$ \\
	\hline
	$L_{2}$ &$1.07\times 10^{-10}$&$1.98\times 10^{-10}$&$4.19\times 10^{-10}$&$8.42\times 10^{-10}$&$1.15\times 10^{-10}$& $8.34\times 10^{-10}$\\
	\hline
\end{tabular}
\label{table Example5.1sol&errorsN=8} 
\end{table}


\begin{table}[H]
\caption {Comparison of numerical errors $E_{9,i}(x)$ $(i=0,1,2,3,4)$ with existing error for Example \ref{Example 5.1}.}
\centering 
\begin{tabular}{|c| c| c |c|c|c|c|} 
	\hline
	$x$&$E_{9,0}(x)$&$E_{9,1}(x)$&$E_{9,2}(x)$&$E_{9,3}(x)$&$E_{9,4}(x)$&BCM \cite{Aourir2025}
	\\[0.5ex] 
	\hline 
	$0.1$ &$7.30\times 10^{-13}$&$6.39\times 10^{-13}$&$1.77\times 10^{-13}$&$4.85\times 10^{-13}$&$1.14\times 10^{-13}$& $-$
	\\
	\hline
	$0.2$ &$1.18\times 10^{-12}$&$1.01\times 10^{-12}$&$2.09\times 10^{-12}$&$4.03\times 10^{-12}$&$2.63\times 10^{-13}$& $-$
	\\
	\hline
	$0.3$ &$8.68\times 10^{-13}$&$1.98\times 10^{-12}$&$9.57\times 10^{-13}$&$3.20\times 10^{-12}$&$1.17\times 10^{-12}$&$-$
	\\
	\hline
	$0.4$ &$5.25\times 10^{-13}$&$1.05\times 10^{-12}$&$1.20\times 10^{-12}$&$2.19\times 10^{-12}$&$2.69\times 10^{-13}$&$-$  
	\\
	\hline
	$0.5$ &$1.15\times 10^{-12}$&$1.27\times 10^{-12}$&$2.70\times 10^{-12}$&$4.92\times 10^{-12}$&$5.98\times 10^{-13}$&$-$ 
	\\
	\hline
	$0.6$ &$5.07\times 10^{-13}$&$1.09\times 10^{-12}$&$1.85\times 10^{-12}$&$5.42\times 10^{-12}$&$1.54\times 10^{-12}$&$-$ 
	\\
	\hline
	$0.7$ &$9.67\times 10^{-13}$&$2.33\times 10^{-12}$&$2.25\times 10^{-12}$&$4.91\times 10^{-12}$&$2.70\times 10^{-13}$&$-$
	\\
	\hline
	$0.8$ &$1.20\times 10^{-12}$&$5.20\times 10^{-13}$&$4.42\times 10^{-12}$&$7.99\times 10^{-12}$&$1.03\times 10^{-12}$&$-$ 
	\\
	\hline
	$0.9$ &$6.17\times 10^{-13}$&$1.35\times 10^{-12}$&$3.75\times 10^{-12}$&$9.89\times 10^{-12}$&$2.07\times 10^{-12}$&$-$ 
	\\
	\hline
	$1.0$ &$2.51\times 10^{-13}$&$1.05\times 10^{-12}$&$5.24\times 10^{-12}$&$1.32\times 10^{-11}$&$2.37\times 10^{-12}$&$-$ 
	\\
	\hline\hline
	$L_{\infty}$ &$1.20\times 10^{-12}$&$2.33\times 10^{-12}$&$5.24\times 10^{-12}$&$1.32\times 10^{-11}$&$2.37\times 10^{-12}$&$2.35\times 10^{-11}$  \\
	\hline
	$L_{2}$  &$2.71\times 10^{-12}$&$4.24\times 10^{-12}$&$9.15\times 10^{-12}$&$2.11\times 10^{-11}$&$3.92\times 10^{-12}$& $-$ \\
	\hline
\end{tabular}
\label{table Example5.1sol&errorsN=9} 
\end{table}

\begin{table}[H]
	\caption {Comparison of numerical solutions $\phi_{14,i}(x)$ $(i=0,1,2,3,4)$ with the exact solution, and other published results for Example \ref{Example 5.2}.}
	\centering 
	\begin{tabular}{|c| c| c |c|c|c|c|c|} 
		\hline
		$x$&$\phi_{{\rm ex.}}(x)$&$\phi_{14,0}(x)$&$\phi_{14,1}(x)$&$\phi_{14,2}(x)$&$\phi_{14,3}(x)$&$\phi_{14,4}(x)$&BCM \cite{Aourir2025}
		\\[0.5ex] 
		\hline 
		$0.1$ &$0.0953102$&$0.0953102$&$0.0953102$&$0.0953102$&$0.0953102$&$0.0953102$&$0.0953102$
		\\
		\hline
		$0.2$ &$0.182322$&$ 0.182322$&$0.182322$&$ 0.182322$&$0.182322$&$ 0.182322$& $0.182322$
		\\
		\hline
		$0.3$ &$0.262364$&$ 0.262364$&$0.262364$&$ 0.262364$&$0.262364$&$ 0.262364$&$0.262364$
		\\
		\hline
		$0.4$ &$0.336472$&$ 0.336472$&$0.336472$&$ 0.336472$&$0.336472$&$ 0.336472$& $0.336472$
		\\
		\hline
		$0.5$ &$0.405465$&$ 0.405465$&$0.405465$&$ 0.405465$&$0.405465$&$ 0.405465$& $0.405465$
		\\
		\hline
		$0.6$ &$0.470004$&$ 0.470004$&$0.470004$&$ 0.470004$&$0.470004$&$ 0.470004$& $0.470004$
		\\
		\hline
		$0.7$ &$0.530628$&$ 0.530628$&$0.530628$&$ 0.530628$&$0.530628$&$ 0.530628$&$0.530628$
		\\
		\hline
		$0.8$ &$0.587787$&$ 0.587787$&$0.587787$&$ 0.587787$&$0.587787$&$ 0.587787$&$0.587787$
		\\
		\hline
		$0.9$ &$0.641854$&$ 0.641854$&$0.641854$&$ 0.641854$&$0.641854$&$ 0.641854$ & $0.641854$
		\\
		\hline
		$1.0$ &$0.693147$&$ 0.693147$&$0.693147$&$ 0.693147$&$0.693147$&$ 0.693147$& $0.693147$
		\\
		\hline
	\end{tabular}
	\label{table Example 5.2 num sol} 
\end{table}

\begin{table}[H]
	\caption {Comparison of numerical errors $E_{14,i}(x)$ $(i=0,1,2,3,4)$ with existing errors for Example \ref{Example 5.2}.}
	\centering 
	\begin{tabular}{|c| c| c |c|c|c|c|} 
		\hline
		$x$&$E_{14,0}(x)$&$E_{14,1}(x)$&$E_{14,2}(x)$&$E_{14,3}(x)$&$E_{14,4}(x)$&BCM \cite{Aourir2025}
		\\[0.5ex] 
		\hline 
		$0.1$ &$5.12\times 10^{-13}$&$1.62\times 10^{-13}$&$1.27\times 10^{-12}$&$2.21\times 10^{-12}$&$4.68\times 10^{-14}$& $1.14\times 10^{-13}$
		\\
		\hline
		$0.2$ &$4.02\times 10^{-14}$&$2.48\times 10^{-13}$&$6.41\times 10^{-13}$&$8.20\times 10^{-13}$&$3.90\times 10^{-13}$& $3.28\times 10^{-13}$
		\\
		\hline
		$0.3$ &$1.79\times 10^{-13}$&$6.01\times 10^{-13}$&$8.51\times 10^{-13}$&$1.92\times 10^{-12}$&$6.61\times 10^{-13}$&$4.83\times 10^{-13}$
		\\
		\hline
		$0.4$ &$1.10\times 10^{-13}$&$1.05\times 10^{-13}$&$1.01\times 10^{-12}$&$1.37\times 10^{-12}$&$5.01\times 10^{-13}$&$4.61\times 10^{-13}$  
		\\
		\hline
		$0.5$ &$3.39\times 10^{-14}$&$1.26\times 10^{-13}$&$1.13\times 10^{-12}$&$2.36\times 10^{-12}$&$6.01\times 10^{-13}$&$3.67\times 10^{-13}$ 
		\\
		\hline
		$0.6$ &$1.61\times 10^{-13}$&$5.28\times 10^{-13}$&$1.18\times 10^{-12}$&$1.83\times 10^{-12}$&$2.66\times 10^{-13}$&$6.01\times 10^{-13}$ 
		\\
		\hline
		$0.7$ &$1.89\times 10^{-13}$&$1.90\times 10^{-13}$&$1.48\times 10^{-12}$&$2.84\times 10^{-12}$&$4.35\times 10^{-13}$&$8.34\times 10^{-13}$
		\\
		\hline
		$0.8$ &$6.21\times 10^{-14}$&$5.43\times 10^{-13}$&$1.69\times 10^{-12}$&$2.67\times 10^{-12}$&$3.21\times 10^{-13}$&$8.38\times 10^{-13}$ 
		\\
		\hline
		$0.9$ &$2.62\times 10^{-13}$&$4.22\times 10^{-13}$&$1.56\times 10^{-12}$&$.32\times 10^{-12}$&$9.15\times 10^{-13}$&$2.94\times 10^{-13}$ 
		\\
		\hline
		$1.0$ &$2.62\times 10^{-13}$&$4.05\times 10^{-13}$&$2.44\times 10^{-12}$&$3.78\times 10^{-12}$&$5.46\times 10^{-13}$&$1.92\times 10^{-11}$ 
		\\
		\hline\hline
		$L_{\infty}$ &$5.12\times 10^{-13}$&$6.01\times 10^{-13}$&$2.44\times 10^{-12}$&$3.78\times 10^{-12}$&$9.15\times 10^{-13}$& $1.92\times 10^{-11}$ \\
		\hline
		$L_{2}$ &$6.66\times 10^{-13}$&$1.19\times 10^{-12}$&$4.47\times 10^{-12}$&$7.79\times 10^{-12}$&$1.64\times 10^{-12}$& $1.92\times 10^{-11}$\\
		\hline
	\end{tabular}
	\label{table Example5.2sol&errorsN=14} 
\end{table}


\begin{table}[H]
	\caption {Comparison of numerical errors $E_{15,i}(x)$ $(i=0,1,2,3,4)$ with existing result for Example \ref{Example 5.2}.}
	\centering 
	\begin{tabular}{|c| c| c |c|c|c|c|} 
		\hline
		$x$&$E_{15,0}(x)$&$E_{15,1}(x)$&$E_{15,2}(x)$&$E_{15,3}(x)$&$E_{15,4}(x)$&BCM \cite{Aourir2025}
		\\[0.5ex] 
		\hline 
		$0.1$ &$2.55\times 10^{-14}$&$5.05\times 10^{-14}$&$1.76\times 10^{-13}$&$3.37\times 10^{-13}$&$5.12\times 10^{-14}$& $-$
		\\
		\hline
		$0.2$ &$9.26\times 10^{-14}$&$1.26\times 10^{-13}$&$1.77\times 10^{-13}$&$2.45\times 10^{-13}$&$8.11\times 10^{-14}$& $-$
		\\
		\hline
		$0.3$ &$1.07\times 10^{-13}$&$1.76\times 10^{-13}$&$6.88\times 10^{-14}$&$1.83\times 10^{-13}$&$9.32\times 10^{-14}$&$-$
		\\
		\hline
		$0.4$ &$1.03\times 10^{-13}$&$1.29\times 10^{-13}$&$2.28\times 10^{-13}$&$3.44\times 10^{-13}$&$6.62\times 10^{-14}$&$-$  
		\\
		\hline
		$0.5$ &$9.77\times 10^{-14}$&$1.73\times 10^{-13}$&$1.11\times 10^{-13}$&$2.61\times 10^{-13}$&$9.93\times 10^{-14}$&$-$ 
		\\
		\hline
		$0.6$ &$8.93\times 10^{-14}$&$1.00\times 10^{-13}$&$2.63\times 10^{-13}$&$3.98\times 10^{-13}$&$7.24\times 10^{-14}$&$-$ 
		\\
		\hline
		$0.7$ &$7.97\times 10^{-14}$&$1.46\times 10^{-13}$&$1.74\times 10^{-13}$&$3.66\times 10^{-13}$&$9.69\times 10^{-14}$&$-$
		\\
		\hline
		$0.8$ &$6.25\times 10^{-14}$&$5.45\times 10^{-14}$&$2.98\times 10^{-13}$&$5.03\times 10^{-13}$&$9.10\times 10^{-15}$&$-$ 
		\\
		\hline
		$0.9$ &$8.77\times 10^{-15}$&$7.00\times 10^{-14}$&$3.31\times 10^{-13}$&$5.11\times 10^{-13}$&$7.58\times 10^{-14}$&$-$ 
		\\
		\hline
		$1.0$ &$1.11\times 10^{-16}$&$4.47\times 10^{-14}$&$2.98\times 10^{-13}$&$6.12\times 10^{-13}$&$1.46\times 10^{-13}$&$-$ 
		\\
		\hline\hline
		$L_{\infty}$ &$1.07\times 10^{-13}$&$1.76\times 10^{-13}$&$3.31\times 10^{-13}$&$6.12\times 10^{-13}$&$1.46\times 10^{-13}$&$2.99\times 10^{-12}$  \\
		\hline
		$L_{2}$  &$2.43\times 10^{-13}$&$3.71\times 10^{-13}$&$7.20\times 10^{-13}$&$1.25\times 10^{-12}$&$2.72\times 10^{-13}$& $-$ \\
		\hline
	\end{tabular}
	\label{table Example5.2sol&errorsN=15} 
\end{table}

\begin{table}[H]
	\caption {Comparison of numerical solutions $\phi_{12,i}(x)$ $(i=0,1,2,3,4)$ with the exact solution, and other published results for Example \ref{Example 5.3}.}
	\centering 
	\begin{tabular}{|c| c| c |c|c|c|c|c|} 
		\hline
		$x$&$\phi_{{\rm ex.}}(x)$&$\phi_{12,0}(x)$&$\phi_{12,1}(x)$&$\phi_{12,2}(x)$&$\phi_{12,3}(x)$&$\phi_{12,4}(x)$&BCM \cite{Aourir2025}
		\\[0.5ex] 
		\hline 
		$0.1$ &$1.01005$&$1.01005$&$1.01005$&$1.01005$&$1.01005$&$1.01005$&$1.01005$
		\\
		\hline
		$0.2$ &$1.04081$&$1.04081$&$1.04081$&$1.04081$&$1.04081$&$1.04081$& $1.04081$
		\\
		\hline
		$0.3$ &$1.09417$&$ 1.09417$&$1.09417$&$ 1.09417$&$1.09417$&$ 1.09417$&$1.09417$
		\\
		\hline
		$0.4$ &$1.17351$&$ 1.17351$&$1.17351$&$ 1.17351$&$1.17351$&$ 1.17351$& $1.17351$
		\\
		\hline
		$0.5$ &$1.28403$&$ 1.28403$&$1.28403$&$ 1.28403$&$1.28403$&$ 1.28403$& $1.28403$
		\\
		\hline
		$0.6$ &$1.43333$&$ 1.43333$&$1.43333$&$ 1.43333$&$1.43333$&$ 1.43333$& $1.43333$
		\\
		\hline
		$0.7$ &$1.63232$&$ 1.63232$&$1.63232$&$ 1.63232$&$1.63232$&$ 1.63232$&$1.63232$
		\\
		\hline
		$0.8$ &$1.89648$&$ 1.89648$&$1.89648$&$ 1.89648$&$1.89648$&$ 1.89648$&$1.89648$
		\\
		\hline
		$0.9$ &$2.24791$&$ 2.24791$&$2.24791$&$ 2.24791$&$2.24791$&$ 2.24791$ & $ 2.24791$
		\\
		\hline
		$1.0$ &$2.71828$&$ 2.71828$&$2.71828$&$ 2.71828$&$2.71828$&$ 2.71828$& $2.71828$
		\\
		\hline
	\end{tabular}
	\label{table Example 5.3 num sol} 
\end{table}

\begin{table}[H]
	\caption {Comparison of numerical errors $E_{12,i}(x)$ $(i=0,1,2,3,4)$ with existing error for Example \ref{Example 5.3}.}
	\centering 
	\begin{tabular}{|c| c| c |c|c|c|c|} 
		\hline
		$x$&$E_{12,0}(x)$&$E_{12,1}(x)$&$E_{12,2}(x)$&$E_{12,3}(x)$&$E_{12,4}(x)$&BCM \cite{Aourir2025}
		\\[0.5ex] 
		\hline 
		$0.1$ &$1.12\times 10^{-10}$&$1.37\times 10^{-10}$&$1.44\times 10^{-10}$&$2.51\times 10^{-10}$&$5.81\times 10^{-11}$& $-$
		\\
		\hline
		$0.2$ &$1.51\times 10^{-10}$&$2.16\times 10^{-10}$&$3.07\times 10^{-11}$&$3.23\times 10^{-11}$&$8.11\times 10^{-11}$& $-$
		\\
		\hline
		$0.3$ &$1.53\times 10^{-10}$&$2.27\times 10^{-10}$&$1.54\times 10^{-10}$&$2.14\times 10^{-10}$&$1.05\times 10^{-10}$&$-$
		\\
		\hline
		$0.4$ &$1.20\times 10^{-10}$&$2.04\times 10^{-10}$&$2.79\times 10^{-11}$&$1.26\times 10^{-10}$&$1.52\times 10^{-10}$&$-$  
		\\
		\hline
		$0.5$ &$2.31\times 10^{-11}$&$2.82\times 10^{-11}$&$5.33\times 10^{-11}$&$5.79\times 10^{-11}$&$1.42\times 10^{-10}$&$-$ 
		\\
		\hline
		$0.6$ &$1.15\times 10^{-10}$&$1.89\times 10^{-10}$&$1.02\times 10^{-10}$&$3.04\times 10^{-10}$&$6.02\times 10^{-11}$&$-$ 
		\\
		\hline
		$0.7$ &$2.01\times 10^{-10}$&$3.24\times 10^{-10}$&$9.20\times 10^{-11}$&$1.12\times 10^{-10}$&$7.52\times 10^{-11}$&$-$
		\\
		\hline
		$0.8$ &$1.98\times 10^{-10}$&$2.65\times 10^{-10}$&$1.64\times 10^{-10}$&$2.00\times 10^{-10}$&$1.34\times 10^{-10}$&$-$ 
		\\
		\hline
		$0.9$ &$1.53\times 10^{-10}$&$2.22\times 10^{-10}$&$3.44\times 10^{-11}$&$3.70\times 10^{-11}$&$3.21\times 10^{-11}$&$-$ 
		\\
		\hline
		$1.0$ &$1.53\times 10^{-11}$&$1.93\times 10^{-11}$&$2.40\times 10^{-10}$&$1.09\times 10^{-10}$&$3.45\times 10^{-10}$&$-$ 
		\\
		\hline\hline
		$L_{\infty}$ &$2.01\times 10^{-10}$&$3.24\times 10^{-10}$&$2.40\times 10^{-10}$&$3.04\times 10^{-10}$&$3.45\times 10^{-10}$& $5.16\times 10^{-9}$ \\
		\hline
		$L_{2}$ &$4.37\times 10^{-10}$&$6.49\times 10^{-10}$&$3.92\times 10^{-10}$&$5.37\times 10^{-10}$&$4.61\times 10^{-10}$& $-$\\
		\hline
	\end{tabular}
	\label{table Example5.3sol&errorsN=12} 
\end{table}


\begin{table}[H]
	\caption {Comparison of numerical errors $E_{14,i}(x)$ $(i=0,1,2,3,4)$ with existing errors for Example \ref{Example 5.3}.}
	\centering 
	\begin{tabular}{|c| c| c |c|c|c|c|} 
		\hline
		$x$&$E_{14,0}(x)$&$E_{14,1}(x)$&$E_{14,2}(x)$&$E_{14,3}(x)$&$E_{14,4}(x)$&BCM \cite{Aourir2025}
		\\[0.5ex] 
		\hline 
		$0.1$ &$9.81\times 10^{-13}$&$3.23\times 10^{-13}$&$2.02\times 10^{-12}$&$4.60\times 10^{-12}$&$7.29\times 10^{-14}$& $1.46\times 10^{-13}$
		\\
		\hline
		$0.2$ &$1.95\times 10^{-13}$&$8.43\times 10^{-13}$&$4.80\times 10^{-13}$&$1.27\times 10^{-13}$&$7.68\times 10^{-13}$& $4.67\times 10^{-13}$
		\\
		\hline
		$0.3$ &$6.44\times 10^{-13}$&$1.46\times 10^{-12}$&$3.62\times 10^{-13}$&$2.29\times 10^{-12}$&$1.21\times 10^{-12}$&$7.03\times 10^{-13}$
		\\
		\hline
		$0.4$ &$5.99\times 10^{-13}$&$1.08\times 10^{-12}$&$7.55\times 10^{-13}$&$1.30\times 10^{-14}$&$1.38\times 10^{-12}$&$8.23\times 10^{-13}$  
		\\
		\hline
		$0.5$ &$1.41\times 10^{-13}$&$3.11\times 10^{-13}$&$3.17\times 10^{-13}$&$2.53\times 10^{-12}$&$1.48\times 10^{-12}$&$1.11\times 10^{-12}$ 
		\\
		\hline
		$0.6$ &$4.55\times 10^{-13}$&$1.03\times 10^{-12}$&$1.74\times 10^{-13}$&$1.10\times 10^{-12}$&$1.22\times 10^{-12}$&$1.39\times 10^{-12}$ 
		\\
		\hline
		$0.7$ &$6.45\times 10^{-13}$&$1.52\times 10^{-12}$&$4.55\times 10^{-13}$&$2.45\times 10^{-12}$&$1.16\times 10^{-12}$&$1.98\times 10^{-12}$
		\\
		\hline
		$0.8$ &$3.18\times 10^{-14}$&$1.04\times 10^{-12}$&$8.47\times 10^{-13}$&$3.39\times 10^{-13}$&$1.84\times 10^{-12}$&$2.44\times 10^{-12}$ 
		\\
		\hline
		$0.9$ &$1.51\times 10^{-12}$&$9.19\times 10^{-13}$&$1.29\times 10^{-12}$&$9.92\times 10^{-13}$&$3.12\times 10^{-12}$&$3.46\times 10^{-13}$ 
		\\
		\hline
		$1.0$ &$1.37\times 10^{-13}$&$5.76\times 10^{-13}$&$2.17\times 10^{-12}$&$8.84\times 10^{-13}$&$3.26\times 10^{-12}$&$5.95\times 10^{-11}$ 
		\\
		\hline\hline
		$L_{\infty}$ &$1.51\times 10^{-12}$&$1.52\times 10^{-12}$&$2.17\times 10^{-12}$&$4.60\times 10^{-12}$&$3.26\times 10^{-12}$&$5.95\times 10^{-11}$  \\
		\hline
		$L_{2}$  &$2.17\times 10^{-12}$&$3.14\times 10^{-12}$&$3.53\times 10^{-12}$&$6.48\times 10^{-12}$&$5.73\times 10^{-12}$& $5.96\times 10^{-11}$ \\
		\hline
	\end{tabular}
	\label{table Example5.3sol&errorsN=14} 
\end{table}

\section{Conclusion}
In this paper, we have developed a collocation method based on shifted Jacobi polynomials for numerical solutions of Volterra-type pantograph integrodifferential equations  arising in real-life phenomena. The significance and scholarly contribution of this research is that the present approach, shifted Jacobi collocation method, has as special cases, five other collocation techniques, namely, shifted Legendre collocation method, shifted Chebyshev collocation method of the first kind, shifted Chebyshev collocation method of the second kind, shifted Chebyshev collocation method of the third kind, and shifted Chebyshev collocation method of the fourth kind. The application of this 5-in-1 method to the pantograph delay integrodifferential equation of Volterra type is a significant contribution of this paper. The applicability and reliability of the present technique was tested and validated with three numerical examples. The results obtained revealed that Jacobi polynomials as basis functions are a reliable and useful family of orthogonal polynomials whose applications are not only restricted within the theory of special functions. The findings of this research highlight this technique as a very efficient and reliable approach for solving Volterra-type integrodifferential equations with proportional delays. The method employed in this research can be extended to fractional order pantograph delay (linear and nonlinear) integrodifferential equations of Volterra type, and of Volterra-Fredholm type. The method can also be applied to higher order pantograph delay integrodifferential equations of Volterra and Fredholm types.

\appendix		

\section{Zeros of Shifted Jacobi Polynomials}

Table \ref{table zeroP8912}  illustrates the zeros $\mathsf{x}^{N}_{q}(\alpha,\beta)$ $(0\leq q\leq N)$  of the  shifted Jacobi polynomials $\mathsf{J}^{(\alpha,\beta)}_{N+1}(x)$ $(N=7,8,11)$ for different values of parameters $\alpha,\beta>-1$.

\begin{table}[H]
	\caption { Explicit values of the zeros $\mathsf{x}_{q}^{N}(\alpha,\beta)$ $(0\leq q\leq N)$ of $\mathsf{J}^{(\alpha,\beta)}_{N+1}(x)$  $(N=7,8,11)$.}
	\centering 
	\begin{tabular}{|c| c| c |c|c|c|} 
		\hline
		$q$ & $\mathsf{x}_{q}^{7}(0,0)$&  $\mathsf{x}_{q}^{7}\left(-\frac{1}{2},-\frac{1}{2} \right) $ &$\mathsf{x}_{q}^{7}\left(\frac{1}{2},\frac{1}{2} \right) $& $\mathsf{x}_{q}^{7}\left(-\frac{1}{2},\frac{1}{2} \right) $ &$\mathsf{x}_{q}^{7}\left(\frac{1}{2},-\frac{1}{2} \right) $
		\\ [0.5ex] 
		\hline 
		$0$ & $0.019855$ & $0.009607$ & $0.030154$  & $0.033764$ & $0.008513$
		\\
		\hline
		$1$ & $0.101667$ & $0.084265$ & $0.116978$  & $0.130496$ & $0.074891$
		\\
		\hline
		$2$ & $0.237234$  &$0.222215$  & $0.250000$  & $0.277131$ & $0.198683$
		\\
		\hline
		$3$ &$0.408283$  & $0.402455$ & $0.413176$  & $0.453866$ & $0.363169$
		\\
		\hline
		$4$ & $0.591717$ & $0.597545$ & $0.586824$ &  $0.636831$ & $0.546134$
		\\
		\hline
		$5$ & $0.762766$ & $0.777785$ & $0.750000$ &  $0.801317 $ & $0.722869$ 
		\\
		\hline
		$6$ & $0.898333$ & $0.915735$  &$0.883022$  & $0.925109$  & $0.869504$
		\\
		\hline
		$7$ &$0.980145$  &$0.990393$  & $0.969846$ & $0.991487$ &  $0.966236$\\
		\hline
		\hline
	\end{tabular}
%
	\begin{tabular}{|c| c| c |c|c|c|} 
		\hline
	$q$ & $\mathsf{x}_{q}^{8}(0,0)$&  $\mathsf{x}_{q}^{8}\left(-\frac{1}{2},-\frac{1}{2} \right) $ &$\mathsf{x}_{q}^{8}\left(\frac{1}{2},\frac{1}{2} \right) $& $\mathsf{x}_{q}^{8}\left(-\frac{1}{2},\frac{1}{2} \right) $ &$\mathsf{x}_{q}^{8}\left(\frac{1}{2},-\frac{1}{2} \right) $
		\\ [0.5ex] 
		\hline 
		$0$ & $0.015920$ & $0.007596$ & $0.024472$  & $0.027091$ & $0.006819$
		\\
		\hline
		$1$ & $0.081984$ & $0.066987$ & $0.095492$  & $0.105430$ & $0.060263$
		\\
		\hline
		$2$ & $0.193314$  &$0.178606$  & $0.206107$  & $0.226526$ & $0.161359$
		\\
		\hline
		$3$ &$0.337873$  & $0.328990$ & $0.345492$  & $0.377257$ & $0.299152$
		\\
		\hline
		$4$ & $0.500000$ & $0.500000$ & $0.500000$ &  $0.541290$ & $0.458710$
		\\
		\hline
		$5$ & $0.662127$ & $0.671010$ & $0.654508$ &  $0.700848 $ & $0.622743$ 
		\\
		\hline
		$6$ & $0.806686$ & $0.821394$  &$0.793893$  & $0.838642$  & $0.773474$
		\\
		\hline
		$7$ &$0.918016$  &$0.933013$  & $0.904508$ & $0.939737$ &  $0.894570$
		\\
		\hline
		$8$ & $0.984080$ & $0.992404$ & $0.975528$ & $0.993181$  &  $0.972909$
		\\
		\hline
		\hline
	\end{tabular}
%
	\begin{tabular}{|c| c| c |c|c|c|} 
		\hline

	$q$ & $\mathsf{x}_{q}^{11}(0,0)$&  $\mathsf{x}_{q}^{11}\left(-\frac{1}{2},-\frac{1}{2} \right) $ &$\mathsf{x}_{q}^{11}\left(\frac{1}{2},\frac{1}{2} \right) $& $\mathsf{x}_{q}^{11}\left(-\frac{1}{2},\frac{1}{2} \right) $ &$\mathsf{x}_{q}^{11}\left(\frac{1}{2},-\frac{1}{2} \right) $
		\\ [0.5ex] 
		\hline 
		$0$ & $0.009220$ & $0.004278$ & $0.010926$  & $0.015708$ & $0.003942$
		\\
		\hline
		$1$ & $0.047941$ & $0.038060$ & $0.057272$  & $0.061847$ & $0.035112$
		\\
		\hline
		$2$ & $0.115049$  &$0.103323$  & $0.125745$  & $0.135516$ & $0.095492$
		\\
		\hline
		$3$ &$0.206341$  & $0.195619$ & $0.215968$  & $0.232087$ & $0.181288$
		\\
		\hline
		$4$ & $0.316084$ & $0.308658$ & $0.322698$ &  $0.345492$ & $0.287110$
		\\
		\hline
		$5$ & $0.437383$ & $0.434737$ & $0.439732$ &  $0.468605$ & $0.406309$ 
		\\
		\hline
		$6$ & $0.562617$ & $0.565263$  &$0.560268$  & $0.593691$  & $0.531395$
		\\
		\hline
		$7$ &$0.683916$  &$0.691342$  & $0.677302$ & $0.712890$ &  $0.654509$
		\\
		\hline
		$8$ & $0.793659$ & $0.804381$ & $0.784032$ & $0.818712$  &  $0.767913$
		\\
		\hline
		$9$ & $0.884951$ & $0.896677$ & $0.874255$  & $0.904508$ & $0.864484$
		\\
		\hline
		$10$ & $0.952059$ & $0.961940$ & $0.942728$  & $0.964888$ & $0.938153$
		\\
		\hline
		$11$ & $0.990780$ & $0.995722$ & $0.985471$  & $0.996057$ & $0.984292$
		\\
		\hline
	\end{tabular}
	\label{table zeroP8912} 
\end{table}

Table \ref{table zeroP1415}  describes the zeros $\mathsf{x}^{N}_{q}(\alpha,\beta)$ $(0\leq q\leq N)$  of the  shifted Jacobi polynomials $\mathsf{J}^{(\alpha,\beta)}_{N+1}(x)$ $(N=13,14)$ for different values of parameters $\alpha,\beta>-1$.

\begin{table}[H]
	\caption { Explicit values of the zeros $\mathsf{x}_{q}^{N}(\alpha,\beta)$ $(0\leq q\leq N)$ of $\mathsf{J}^{(\alpha,\beta)}_{N+1}(x)$  $(N=13,14)$.}
	\centering 
	\begin{tabular}{|c| c| c |c|c|c|} 
		\hline
	$q$ & $\mathsf{x}_{q}^{13}(0,0)$&  $\mathsf{x}_{q}^{13}\left(-\frac{1}{2},-\frac{1}{2} \right) $ &$\mathsf{x}_{q}^{13}\left(\frac{1}{2},\frac{1}{2} \right) $& $\mathsf{x}_{q}^{13}\left(-\frac{1}{2},\frac{1}{2} \right) $ &$\mathsf{x}_{q}^{13}\left(\frac{1}{2},-\frac{1}{2} \right) $
		\\ [0.5ex] 
		\hline 
		$0$ & $0.006858$ & $0.003144$ & $0.010926$  & $0.011690$ & $0.002931$
		\\
		\hline
		$1$ & $0.035783$ & $0.028058$ & $0.043227$  & $0.046212$ & $0.026173$
		\\
		\hline
		$2$ & $0.086399$  &$0.076638$  & $0.095492$  & $0.101953$ & $0.071571$
		\\
		\hline
		$3$ &$0.156354$  & $0.146447$ & $0.165435$  & $0.176307$ & $0.137002$
		\\
		\hline
		$4$ & $0.242376$ & $0.233984$ & $0.250000$ &  $0.265796$ & $0.219406$
		\\
		\hline
		$5$ & $0.340444$ & $0.334860$ & $0.345491$ &  $0.366236$ & $0.314931$ 
		\\
		\hline
		$6$ & $0.445973$ & $0.444018$  &$0.447736$  & $0.472930$  & $0.419109$
		\\
		\hline
		$7$ &$0.554027$  &$0.555982$  & $0.552264$ & $0.580891$ &  $0.527069$
		\\
		\hline
		$8$ & $0.659556$ & $0.665140$ & $0.654509$ & $0.685069$  &  $0.633764$
		\\
		\hline
		$9$ & $0.757624$ & $0.766015$ & $0.749999$  & $0.780594$ & $0.734204$
		\\
		\hline
		$10$ & $0.843646$ & $0.853555$ & $0.834566$  & $0.862997$ & $0.823693$
		\\
		\hline
		$11$ & $0.913601$ & $0.923360$ & $0.904508$  & $0.928429$ & $0.898046$
		\\
		\hline
		$12$ & $0.964217$ & $0.971943$ & $0.956773$  & $0.973826$ & $0.953788$
		\\
		\hline
		$13$ & $0.993142$ & $0.996856$ & $0.989074$  & $0.997069$ & $0.988310$
		\\
		\hline
		\hline
	\end{tabular}
%
	\begin{tabular}{|c| c| c |c|c|c|} 
		\hline
$q$ & $\mathsf{x}_{q}^{14}(0,0)$&  $\mathsf{x}_{q}^{14}\left(-\frac{1}{2},-\frac{1}{2} \right) $ &$\mathsf{x}_{q}^{14}\left(\frac{1}{2},\frac{1}{2} \right) $& $\mathsf{x}_{q}^{14}\left(-\frac{1}{2},\frac{1}{2} \right) $ &$\mathsf{x}_{q}^{14}\left(\frac{1}{2},-\frac{1}{2} \right) $
		\\ [0.5ex] 
		\hline 
		$0$ & $0.006004$ & $0.002739$ & $0.009607$  & $0.010235$ & $0.002565$
		\\
		\hline
		$1$ & $0.031363$ & $0.024472$ & $0.038060$  & $0.040521$ & $0.022930$
		\\
		\hline
		$2$ & $0.075897$  &$0.066987$  & $0.084265$  & $0.089618$ & $0.062827$
		\\
		\hline
		$3$ &$0.137791$  & $0.128428$ & $0.146447$  & $0.155516$ & $0.120621$
		\\
		\hline
		$4$ & $0.214514$ & $0.206107$ & $0.222215$ &  $0.235518$ & $0.193947$
		\\
		\hline
		$5$ & $0.302924$ & $0.296632$ & $0.308658$ &  $0.326347$ & $0.279802$ 
		\\
		\hline
		$6$ & $0.399403$ & $0.396044$  &$0.402455$  & $0.424288$  & $0.374674$
		\\
		\hline
		$7$ &$0.500000$  &$0.500000$  & $0.500000$ & $0.525322$ &  $0.474675$
		\\
		\hline
		$8$ & $0.600597$ & $0.603956$ & $0.597547$ & $0.625331$  &  $0.575714$
		\\
		\hline
		$9$ & $0.697076$ & $0.703368$ & $0.691336$  & $0.720189$ & $0.673653$
		\\
		\hline
		$10$ & $0.785486$ & $0.793893$ & $0.777794$  & $0.806064$ & $0.764480$
		\\
		\hline
		$11$ & $0.862209$ & $0.871572$ & $0.853541$  & $0.879365$ & $0.844487$
		\\
		\hline
		$12$ & $0.924103$ & $0.933013$ & $0.915747$  & $0.937187$ & $0.910378$
		\\
		\hline
		$13$ & $0.968637$ & $0.975528$ & $0.961932$  & $0.977059$ & $0.959481$
		\\
		\hline
		$14$ & $0.993996$ & $0.997261$ & $0.990395$  & $0.997438$ & $0.989764$
		\\
		\hline	\end{tabular}
	\label{table zeroP1415} 
\end{table}

			%
			%
			
\textbf{Conflicts of Interest:} The authors declare that there is no conflict of interest.
		
\textbf{Ethical Approval:}  This article does not contain any studies with human participants or animals performed by the authors. 

\textbf{Data Availability Statement:} No data are associated with this manuscript.

\textbf{Funding:} No funding was received for this research.

		\end{document}